\documentclass[a4paper,10pt]{amsart}
\usepackage[english]{babel}
\usepackage[utf8]{inputenc}
\usepackage[T1]{fontenc}
\usepackage{csquotes}
\usepackage[style=numeric,%
	useprefix,%
	giveninits=true,%
	hyperref,%
	doi=false,%
	url=false,%
	isbn=false,%
	backend=bibtex,%
	maxbibnames=99%
	]{biblatex}
\bibliography{./BIB}

\usepackage{amssymb}
\usepackage{mathrsfs}
\usepackage{hyperref}
\usepackage[usenames,dvipsnames]{xcolor}
\hypersetup{colorlinks,%
citecolor=Black,%
filecolor=Black,%
linkcolor=Black,%
urlcolor=Black}
\usepackage{enumitem}	
\usepackage{array}		
\usepackage{tikz-cd}
\newcommand{\scr}[1]{\mathscr{#1}}
\newcommand{\frk}[1]{\mathfrak{#1}}
\newcommand{\bb}[1]{\mathbb{#1}}
\newcommand{\cal}[1]{\mathcal{#1}}
\newcommand{\N}{\mathbb{N}}	
\newcommand{\R}{\mathbb{R}}	
\newcommand{\Id}{\mathrm{Id}}	
\newcommand{\Span}{\mathrm{span}}	
\newcommand{\dd}{\,\mathrm{d}}	
\newcommand{\de}{\partial}		
\newcommand{\THEN}{\Rightarrow}	

\newcommand{\J}{\mathtt{J}}
\newcommand{\ad}{\mathrm{ad}}

\newcommand{\HH}{\bb H}
\newcommand{\Ad}{\operatorname{Ad}}

\newcommand{\vol}{\mathtt{vol}}
\newcommand{\Jac}{\mathrm{Jac}}
\newcommand{\Alt}{\mathrm{Alt}}
\newcommand{\MC}{\mathtt{MC}}
\newcommand{\trace}{\mathrm{trace}}
\newcommand{\sre}{\mathtt{SExp}}
\newcommand{\so}{\mathfrak{so}}

\newcommand{\Kmin}{\Gamma}
\newcommand{\Kmax}{\hat\Gamma}

\newcommand{\wt}{\widetilde}
\newcommand{\E}{\mathtt{E}}
\newcommand{\diag}{\mathrm{diag}}
\newcommand{\Sing}{\mathtt{Sing}}

\usepackage{dsfont}

   \def\XXint#1#2#3{{\setbox0=\hbox{$#1{#2#3}{\int}$}
        \vcenter{\hbox{$#2#3$}}\kern-.5\wd0}}

\theoremstyle{plain}
\newtheorem{proposition}{Proposition}[section]
\newtheorem{theorem}[proposition]{Theorem}
\newtheorem{lemma}[proposition]{Lemma}
\newtheorem{corollary}[proposition]{Corollary}
\newtheorem{thm}{Theorem}[section]

\theoremstyle{definition}
\newtheorem{definition}[proposition]{Definition}
\newtheorem{remark}[proposition]{Remark}

\theoremstyle{remark}

\title[CE and GEO dimension]{Curvature exponent 
and geodesic dimension
on Sard-regular Carnot groups}
\author[Nicolussi~Golo]{Sebastiano Nicolussi Golo}
\address[Nicolussi~Golo]{Department of Mathematics and Statistics, 40014 University of Jyväskylä, Finland}	
\email{sebastiano@nicolussigolo.eu}

\author[Zhang]{Ye Zhang}
\address[Zhang]{Analysis on Metric Spaces Unit, Okinawa Institute of Science and Technology Graduate University, Okinawa 904-0495, Japan}
\email{Ye.Zhang2@oist.jp}
\date{\today. \IfFileExists{./.gittex}{\input{./.gittex}}{}}
\keywords{Carnot groups, Curvature exponent, Geodesic dimension, Sub-Riemannian geometry}
\subjclass[2020]{53C17, 53C23}

\begin{document}
\maketitle

\begin{abstract}
In this paper we characterize the geodesic dimension $N_{GEO}$ 
and give a new lower bound to the curvature exponent $N_{CE}$ on Sard-regular Carnot groups. 
As an application, we give an example of step-two Carnot group on which $N_{CE} > N_{GEO}$:
this answers a question posed by Rizzi in~\cite{MR3502622}.
\end{abstract}

\setcounter{tocdepth}{2}
\phantomsection
\addcontentsline{toc}{section}{Contents}
\tableofcontents


\section{Introduction}
\subsection{Geodesic Dimension and Curvature Exponent}
In a geodesic metric measure space $(G,d,\vol)$, 
for a point $p\in G$, a set $E\subset G$, and $t\in[0,1]$,
define the \emph{set of $t$-intermediate points}
\begin{equation}\label{eq64c38a13}
Z_t(p,E) := \left\{ z\in G: \exists q\in E: d(p,z) = t d(p,q), d(z,q) = (1-t)d(p,q) \right\} .
\end{equation}
Although our definition of set of intermediate points is not the same as in \cite{MR3852258,MR3502622},
we will clarify in Remark~\ref{rem64c3884d} that they are interchangeable in our study of geodesic dimension and curvature exponent.

We are interested in the behaviour of $t\mapsto \vol(Z_t(p,E))$ for $E$ measurable with $0<\vol(E)<\infty$.
In particular, we have two characteristic exponents.
First, the \emph{geodesic dimension at $p$} is 
\[
N_{GEO}(p):=
\inf\left\{ N>0 :
\sup_{E\in\scr F} \limsup_{t\to0} \frac{\vol(Z_t(p,E))}{t^N\vol(E)} = \infty 
\right\}
\]
where $\scr F := \{E\subset G\text{ bounded, measurable with }0<\vol(E)<\infty\}$.

Second, the \emph{curvature exponent at $p$} is
\[
N_{CE}(p)
:= \inf\left\{N>0 :  \vol(Z_t(p,E)) \ge t^N \vol(E), \forall t\in[0,1],\forall E\in\scr F \right\} .
\]
The geodesic dimension was originally introduced in~\cite{MR3852258}.
The curvature exponent was originally introduced in~\cite{MR3110060}.
See also \cite{MR3502622}. 
Note that we have $N_{CE}(p) \ge N_{GEO}(p)$ by definition.

If $G$ is Lie group, $d$ is left-invariant and $\vol$ is a Haar measure, 
the choice of $p$ does not play any role and thus we can focus on $p=e$, the identity element of $G$.
Consequently, we have the \emph{geodesic dimension} $N_{GEO}$ of $G$ and
the \emph{curvature exponent} $N_{CE}$ of $G$.

\subsection{Sard-regular Carnot group}
We will give estimates for the geodesic dimension and the curvature exponent of Carnot groups.
Let $G$ be a Carnot group with stratified Lie algebra $\frk g = \bigoplus_{j=1}^s V_j$
and a fixed scalar product $\langle \cdot,\cdot \rangle$ on $V_1$.
See Section~\ref{sec6492eb72} for details.

The sub-Riemannian exponential map (based at the identity element $e$)
$\sre$ is an analytic function from $T_e^*G = \frk g^*$ to $G$, where $\frk g^*$ is the dual of the Lie algebra~$\frk g$.
We denote by $\Jac(\sre)$ the Jacobian determinant of $\sre$.
For more details, see Sections~\ref{sec6492cb92} and \ref{sec6492eb72}.

\begin{definition}\label{def6401a951}
	Define $\scr D\subset\frk g^*$ as the open set of all $\xi\in\frk g^*$ such that
	$\Jac(\sre)(t\xi)\neq0$ for all $t\in(0,1]$ and $t\mapsto\sre(t\xi)$ is the unique constant-speed length-minimizing curve $[0,1]\to G$ from $e$ to $\sre(\xi)$.
	
	Next, define $\cal S_e\subset G$ as the image $\sre(\scr D)$.
	Notice that $\sre$ is a diffeomorphism from $\scr D$ to $\cal S_e$.
	We denote by $\rho:\cal S_e\to\scr D$ the inverse of $\sre$ on $\cal S_e$.
\end{definition}

It is well known that $\cal S_e$ is dense in $G$, 
see~\cite{MR2513150} and references therein.
However,
it is not known whether $\cal S_e$ has always full measure in $G$,
see for instance~\cite{MR3569245}.
In our study, we need $\cal S_e$ to have full measure.

\begin{definition}[Sard-regular]\label{def6401a973}
	We say that a Carnot group $G$ is \emph{Sard-regular}
	if the set $\cal S_e$ has full measure in $G$. 
\end{definition}

Note that Carnot groups of step two are Sard-regular by~\cite[Proposition~15]{MR3110060}.
We stress that 
it is an open question whether all Carnot groups are Sard-regular,
that is, this hypothesis might be superfluous.

\subsection{Sub-Riemannian exponential map}

We will use the fact that the sub-Riemannian exponential map is analytic to estimate both $N_{GEO}$ and $N_{CE}$.

The stratification $\frk g= \bigoplus_{j=1}^sV_j$ of the Lie algebra of $G$
induces a splitting $\frk g^* = \bigoplus_{j=1}^sV_j^*$
of the dual space $\frk g^*$,
where $V_j^* = \{\alpha\in\frk g^*:V_i\subset\ker\alpha,\ \forall i\neq j\}$.

For $\lambda\in\R$, define $\zeta_\lambda:\frk g^*\to\frk g^*$ as
\[
\zeta_\lambda\left( \sum_{j=1}^s \xi_j \right) = \sum_{j=1}^s \lambda^{j-1} \xi_j .
\]
Notice that, when $\lambda\to0$, we have $\zeta_\lambda(\xi)\to\xi_1\in V_1^*$.
By insight to the Hamiltonian system, we obtain in Section~\ref{sec6425df39} that the Jacobian determinant of the sub-Riemannian exponential map satisfies
\begin{equation}\label{eq64009980}
	\Jac(\sre)(\lambda\xi) = \lambda^{2Q-2n} \Jac(\sre)(\zeta_\lambda(\xi)) .
\end{equation}
Here $n =\dim(G)$ and $Q =\sum_{j=1}^sj\dim(V_j)$ are the topological and homogeneous dimensions of $G$, respectively.
In Proposition~\ref{prop64255e1c}, we prove that, if $G$ is Sard-regular, then,
for every measurable $E\subset G$,
\begin{equation}\label{eq6401ad3d}
	\vol(Z_\lambda(e,E)) = \lambda^{2Q-n}  \int_{\rho(E \cap \cal S_e)} | \Jac(\sre)(\zeta_\lambda(\xi)) | \dd\xi .
\end{equation}
Formula~\eqref{eq6401ad3d} is crucial for our estimates of $N_{GEO}$ and $N_{CE}$.

By analyticity of the sub-Riemannian exponential map,
there are analytic functions $P_k:\frk g^*\to\R$
such that for every $\xi\in\frk g^*$ there exists $\lambda_\xi>0$ with
\begin{equation}\label{eq64255fe9}
\Jac(\sre)(\zeta_\lambda(\xi))
= \sum_{k=0}^\infty P_k(\xi) \lambda^k ,
\end{equation}
for $|\lambda|<\lambda_\xi$.
One can take $\xi\mapsto\lambda_\xi$ continuous.
Define
\begin{equation}\label{eq6400a110}
\begin{aligned}
	\Kmin(\xi) &:= \min\{k:P_k(\xi)\neq0\} ,  \\
	\Kmin(G) &:= \min\{\Kmin(\xi) : \xi\in\frk g^*\} , \text{ and} \\
	\Kmax(G) &:= \sup\{\Kmin(\xi): \xi\in\frk g^* \text{ with }\Kmin(\xi)<\infty\} .
\end{aligned}
\end{equation}

\subsection{Main Results}
Our main results are the following two theorems~\ref{thm6401aa05} and~\ref{thm6419896b},
which we then summarize in Theorem~\ref{thm6401b142}.

\begin{thm}\label{thm6401aa05}
	If $G$ is a Sard-regular Carnot group, then
	\[
	N_{GEO} = 2Q-n+\Kmin(G) ,
	\]
	where $n=\dim(G)$ is the topological dimension,
	$Q=\sum_{j=1}^sj\dim(V_j)$ is the homogeneous dimension,
	and $\Kmin(G)$ is defined in~\eqref{eq6400a110}.
\end{thm}

See Section~\ref{sec6425dfb4} for the proof of Theorem~\ref{thm6401aa05}.

\begin{thm}\label{thm6419896b}
	In a Sard-regular Carnot group, we have
	\[
	2Q-n+\Kmax(G) \le N_{CE}
	\]
	where $n=\dim(G)$ is the topological dimension,
	$Q=\sum_{j=1}^sj\dim(V_j)$ is the homogeneous dimension,
	while $\Kmax(G)$ is defined in~\eqref{eq6400a110}.
\end{thm}

See Section~\ref{sec6425e252} for a proof of Theorem~\ref{thm6419896b}.

We summarize the results of Theorems~\ref{thm6401aa05} and~\ref{thm6419896b}
in the following statement:
\begin{thm}\label{thm6401b142}
	In a Sard-regular Carnot group, we have
	\[
	n \le Q \le  N_{GEO} = 2Q-n+\Kmin(G) \le 2Q-n+\Kmax(G) \le N_{CE} ,
	\]
	where $n=\dim(G)$ is the topological dimension,
	$Q=\sum_{j=1}^sj\dim(V_j)$ is the homogeneous dimension,
	while $\Kmin(G)$ and $\Kmax(G)$ are defined in~\eqref{eq6400a110}.
\end{thm}

It is known that $N_{CE}$ is finite on ideal Carnot groups by Rifford in~\cite{MR3110060} and by Barilari--Rizzi in~\cite{MR3935035}. Then it was generalized to the class of so-called Lipschitz Carnot groups, 
which includes step-two Carnot groups, see~\cite{BR20}. 
The fact that $Q \le N_{GEO}\le N_{CE}$ was already known, see~\cite{MR3502622} or \cite[Proposition~5.49]{MR3852258}.
To our best knowledge,
all known examples of sub-Riemannian Carnot groups satisfy $N_{GEO}= N_{CE}$.
In particular, Juillet showed in \cite{J09} that $N_{CE} = N_{GEO} = 2Q - n$ on the Heisenberg group $\bb H^n$.
Later on, the equality $N_{GEO}= N_{CE}$ has been proven for all corank $1$ Carnot groups in \cite{MR3502622}
and for generalized H-type groups in \cite{BR18}.

In Carnot groups of step two, we will give a constructive method to compute both 
$\Kmin(G)$ and $\Kmax(G)$.
As a consequence, we will provide examples of Carnot groups of step two where $\Kmin(G)<\Kmax(G)$.
In such cases we have $N_{GEO} < N_{CE}$, which answers a question posed by Rizzi in~\cite{MR3502622}.

\begin{corollary}\label{mainc}
	There are sub-Riemannian Carnot groups of step 2 where $N_{GEO}<N_{CE}$.
\end{corollary}

Borza--Tashiro 
have recently given
examples of sub-Finsler Carnot groups with $N_{GEO} < N_{CE}$, see \cite{BT23}.
Similarly with what we do, they study asymptotic behaviours of the Jacobian of the sub-Riemannian, or sub-Finsler, exponential map.
In their case, since they use $\ell^p$-norms instead of the $\ell^2$-norm we use, 
the value of $\Gamma(\xi)$ may be fractional.

\begin{remark}\label{rel1}
After this paper was completed, Rizzi informed us of the following results in \cite{MR3852258}. In our framework of Sard-regular Carnot groups, it follows from \cite[Lemma 6.27]{MR3852258} that our $2Q +n +\Kmin(\lambda)$ in this paper coincides with $\cal N_\lambda$ there, whose value by \cite[Definition 5.44]{MR3852258} is given by geodesic growth vector $\cal G_\lambda$. We refer to Remark \ref{rel2} below for more details about the value of $\cal N_\lambda$ on step-two Carnot groups. Furthermore, in \cite[Definition 5.47]{MR3852258} the geodesic dimension is actually defined by the minimum of those $\cal N_\lambda$, which is exactly our Theorem \ref{thm6401aa05}. However, our Theorem \ref{thm6419896b} and Corollary \ref{mainc} remain new.
\end{remark}

\subsection{Summary}
In Section~\ref{sec6492cb92}, we give a brief description of the Hamiltonian formalism that defines the sub-Riemannian exponential map.
We then introduce Carnot groups in Section~\ref{sec6492eb72}.
Section~\ref{sec6425dfb4} contains the proof of Theorem~\ref{thm6401aa05}, 
while Section~\ref{sec6425e252} the proof of Theorem~\ref{thm6419896b}.
In Section~\ref{sec6496dd08}, we study more closely Carnot groups of step two, and the sub-Riemannian exponential map thereof.
Finally, in Section~\ref{sec64c3b071} we compute several explicit examples.


\section{Hamiltonian systems on Lie groups}\label{sec6492cb92}
In this section, $G$ denotes a Lie group with Lie algebra $\frk g$.
For sake of completeness, we will describe the standard construction of Hamiltonian systems on $G$
given by left-invariant Hamiltonians $H:T^*G\to\R$.
We will then apply this formalism to sub-Riemannian Carnot groups in the next section.

We identify $\frk g$ with the tangent space $T_eG$ of $G$ at the identity element $e\in G$.
For any function $v:U\to\frk g$ on an open subset $U\subset G$, we define the vector field $\tilde v\in\Gamma(TU)$ on $U$ as
\[
\tilde v(p) := DL_p|_e[v] \in T_pG .
\]
Similarly, if $\alpha:U\to\frk g^*$, we define $\tilde\alpha\in\Gamma(T^*U)$ as
\[
\tilde\alpha(p) = DL_{p^{-1}}|_p^*[\alpha] \in T_p^*G .
\]
Notice that the vector field $\tilde v$ is left-invariant if and only if the function $v$ is constant, and similarly $\tilde\alpha$ is left-invariant if and only if $\alpha$ is constant.

We will denote by $\langle \cdot|\cdot \rangle$ the pairing of a vector space with its dual, or, more generally, the pairing between linear maps and their domain.
Scalar products are usually denoted by $\langle \cdot,\cdot \rangle$.

\subsection{Differential forms}\label{ssec64117def}
For a vector space $V$ and an open set $U\subset G$,
we define
\[
\Omega_L^k(U;V) := C^\infty(U;\Alt^k(\frk g;V))
\]
where $\Alt^k(\frk g;V)$ is the space of $k$-multilinear alternating maps from $\frk g$ to $V$.
Elements in $\Omega_L^k(U;V)$ are identified with differential forms on $U$ as follows.
Define $\MC:\Omega^k(U;V) \to \Omega_L^k(U;V)$ by
\[
\langle \MC(\tilde\alpha)(p) | v_1\wedge\dots\wedge v_k \rangle
= \langle \tilde\alpha(p) | \tilde v_1(p)\wedge\dots\wedge \tilde v_k(p) \rangle ,
\]
for $p\in U$, $\tilde\alpha\in\Omega^k(U;V)$ and $v_j\in \frk g$.
Vice versa, if $\alpha\in\Omega_L^k(U;V)$, we denote by $\tilde\alpha$ the only element of $\Omega^k(U;V)$ such that $\MC(\tilde\alpha) = \alpha$.

We use the map $\MC$ to push the exterior derivative from $\Omega^k(U;V)$ to $\Omega_L^k(U;V)$.
We define $d:\Omega_L^k(U;V)\to\Omega_L^{k+1}(U;V)$ as $d\alpha := \MC(d\tilde\alpha)$.
Using standard formulas for the exterior differential, we obtain 
for $\alpha\in\Omega_L^k(U;V)$ and $v_0,\dots,v_k\in\frk g$,
\begin{equation}\label{eq63d1a655}
\begin{aligned}
\langle d\alpha | v_0\wedge\dots\wedge v_k \rangle
	&= \sum_{j=0}^k (-1)^j \tilde v_j \langle \alpha(\cdot)| v_0\wedge\dots\wedge\hat v_j \wedge\dots\wedge v_k \rangle  \\
	&
+ \sum_{i<j}^k (-1)^{i+j} \langle \alpha | [v_i,v_j]\wedge v_0\wedge\dots\wedge\hat v_i \wedge\dots\wedge\hat v_j \wedge\dots\wedge v_k \rangle .
\end{aligned}
\end{equation}

\subsection{The cotangent bundle and Hamiltonian mechanics}
The cotangent bundle $T^*G$ of a Lie group $G$ with Lie algebra $\frk g$ has a (left-)canonical group structure as direct product $G\times\frk g^*$, where $\frk g^*$ is seen as abelian Lie group.
More precisely, 
we make $T^*G$ into a Lie group isomorphic to $G\times\frk g^*$
via a map $\Phi_L:G\times\frk g^* \to T^*G$ defined by
\[
\Phi_L(g,\alpha) := \tilde\alpha(g) = DL_{g^{-1}}|_g^*[\alpha] .
\]
This group structure allows us to use the notation from Section~\ref{ssec64117def} for differential forms on $T^*G$.

On $T^*G$, we have the tautological 1-form $\tau\in\Omega^1(T^*G)$,
\[
\tau(\xi)[w] = \langle \xi|D\pi_{T^*G}w \rangle ,
\qquad \text{ for }\xi\in T^*G\text{ and } w\in T_\xi(T^*G) ,
\]
where $\pi_{T^*G}:T^*G\to G$ is the bundle projection.
We pull back $\tau$ to $G\times\frk g^*$ via $\Phi_L$ and we take its left version $\tau_L$.
In other words, we define $\tau_L\in\Omega_L^1(G\times\frk g^*;\R)$ as
\[
\tau_L = \MC(\Phi_L^*\tau).
\]
It might look abstract and complicated, but the point of this reasoning is to get the following formula right, that is, we really want to be sure that we are dealing with the standard tautological form and later with the standard symplectic form.
Indeed, the above formula and the definition of the exterior derivative on $\Omega_L^1(G\times\frk g^*;\R)$ imply 
\[
\omega_L := - \dd\tau_L = - \MC(\Phi_L^*\dd\tau) .
\]

A short computation gives us, for all $(g,\alpha,v,\mu)\in G\times\frk g^*\times\frk g\times\frk g^*$,
\begin{equation}\label{eq63d1ab31}
	\langle \tau_L(g,\alpha) | (v,\mu) \rangle = \langle \alpha|v \rangle .
\end{equation}
Indeed,
\begin{align*}
	\langle \tau_L(g,\alpha) | (v,\mu) \rangle 
	&= \langle \Phi_L^*\tau(g,\alpha) | 
		\left.\frac{\dd}{\dd\epsilon}\right|_{\epsilon=0} 
		(g\exp(\epsilon v),\alpha+\epsilon\mu)
		\rangle \\
	&= \langle \tau(\tilde\alpha(g)) | 
		\left.\frac{\dd}{\dd\epsilon}\right|_{\epsilon=0} 
		\Phi_L(g\exp(\epsilon v),\alpha+\epsilon\mu)
		\rangle \\
	&= \langle \tilde\alpha(g) | 
		D\pi_{T^*G}
		\left.\frac{\dd}{\dd\epsilon}\right|_{\epsilon=0} 
		\Phi_L(g\exp(\epsilon v),\alpha+\epsilon\mu)
		\rangle \\
	&= \langle \tilde\alpha(g) | 
		\left.\frac{\dd}{\dd\epsilon}\right|_{\epsilon=0} 
		\pi_{T^*G}
		\Phi_L(g\exp(\epsilon v),\alpha+\epsilon\mu)
		\rangle \\
	&= \langle \tilde\alpha(g) | 
		\left.\frac{\dd}{\dd\epsilon}\right|_{\epsilon=0} 
		g\exp(\epsilon v)
		\rangle \\
	&= \langle \tilde\alpha(g) | \tilde v(g) \rangle 
	= \langle \alpha|v \rangle .
\end{align*}

Now we can compute the symplectic form $\omega_L = - d\tau_L$ 
as
\begin{equation}
	\langle \omega_L(g,\alpha) | (v_0,\mu_0)\wedge(v_1,\mu_1) \rangle
	= \langle \mu_1|v_0 \rangle - \langle \mu_0|v_1 \rangle
		+ \langle \alpha | [v_0,v_1] \rangle .
\end{equation}
Indeed, using~\eqref{eq63d1a655}, we easily compute
\begin{align*}
	\langle \omega_L(g,\alpha) | (v_0,\mu_0)\wedge(v_1,\mu_1) \rangle
	&= - (v_0,\mu_0)\tilde{} \langle \tau_L(\cdot) |  (v_1,\mu_1)\rangle
		+ (v_1,\mu_1)\tilde{} \langle \tau_L(\cdot) |  (v_0,\mu_0)\rangle \\
	&\qquad	+ \langle \tau_L(g,\alpha) | [(v_0,\mu_0),(v_1,\mu_1)] \rangle \\
	&= - \langle \mu_0|v_1 \rangle + \langle \mu_1|v_0 \rangle
		+ \langle \alpha | [v_0,v_1] \rangle .
\end{align*}

If $H:G\times\frk g^*\to\R$ is a smooth function, we define 
$\cal X_H:G\times\frk g^*\to \frk g\times\frk g^*$ by the formula
\begin{equation}\label{eq63d1b115}
	\langle \omega_L(g,\alpha) | \cal X_H(g,\alpha)\wedge(v_1,\mu_1) \rangle
	= \left.\frac{\dd}{\dd\epsilon}\right|_{\epsilon=0} H(g\exp(\epsilon v_1),\alpha+\epsilon\mu_1) ,
\end{equation}
which is required to hold for all $(g,\alpha,v_1,\mu_1)\in G\times\frk g^*\times\frk g\times\frk g^*$.
To compute $\cal X_H$, we write $\cal X_H = (v_H,\mu_H)$
with $v_H:G\times\frk g^*\to \frk g$ and $\mu_H:G\times\frk g^*\to \frk g^*$.
By linearity, we obtain that~\eqref{eq63d1b115} is equivalent to 
\begin{equation}\label{eq63d1b1ae}
\begin{cases}
\langle \mu_1|v_H \rangle
		&= \left.\frac{\dd}{\dd\epsilon}\right|_{\epsilon=0} H(g,\alpha+\epsilon\mu_1) , \\
	\langle \mu_H|v_1 \rangle - \langle \alpha | [v_H,v_1] \rangle
		&= -\left.\frac{\dd}{\dd\epsilon}\right|_{\epsilon=0} H(g\exp(\epsilon v_1),\alpha) , 
	 \end{cases}
\end{equation}
for all $(g,\alpha,v_1,\mu_1)\in G\times\frk g^*\times\frk g\times\frk g^*$.

A solution to the \emph{Hamiltonian equations} is a curve $t\mapsto (g(t),\alpha(t))$ such that
\begin{equation}\label{eq63d3a6ab}
	\begin{cases}
	DL_{g(t)}^{-1}\dot g(t) &= v_H(g(t),\alpha(t)) , \\
	\dot \alpha(t) &= \mu_H(g(t),\alpha(t)) .
	\end{cases}
\end{equation}

\subsection{Sub-Riemannian Hamiltonian system}
Let $V_1\subset\frk g$ be a bracket-generating linear subspace of $\frk g$ and $\langle \cdot,\cdot \rangle$ a scalar product on $V_1$.
The scalar product on $V_1$ induces a scalar product $\langle \cdot,\cdot \rangle^*$ on the dual space $V_1^*$.
We will use the standard notation $\alpha\mapsto\alpha^\sharp$ to denote the canonical isomorphism $V_1^*\to V_1$ induced by the scalar product, 
and its inverse $V_1\to V_1^*$, 
$v\mapsto v^\flat$. 
For example, by definition for every $\alpha,\beta\in V_1^*$ we have
$\alpha^\sharp ,\beta^\sharp \in V_1$ with
\[
\langle \alpha,\beta \rangle^*
= \langle \alpha | \beta^\sharp \rangle 
= \langle \alpha^\sharp ,\beta^\sharp  \rangle.
\]

The Hamiltonian we are interested in is 
\begin{equation}\label{eq6412f2b5}
 H:G\times\frk g^*\to\R,
\qquad  H( g,\alpha ) = \frac12 \langle \alpha|_{V_1},\alpha|_{V_1} \rangle^* .
\end{equation}
We have, for all $(g,\alpha,v,\mu) \in G\times\frk g^*\times\frk g\times\frk g^*$,
\begin{align*}
	&\left.\frac{\dd}{\dd\epsilon}\right|_{\epsilon=0} H(g\exp(\epsilon v),\alpha)
	= 0 ,\text{ and}\\
	&\left.\frac{\dd}{\dd\epsilon}\right|_{\epsilon=0} H(g,\alpha+\epsilon\mu)
	= \langle \alpha|_{V_1},\mu|_{V_1} \rangle^* .
\end{align*}
Thus, equations~\eqref{eq63d1b1ae} defining the Hamiltonian vector field $\cal X_H$ become
\[
\begin{cases}
	 \langle \mu|v_H \rangle
		&= \langle \alpha|_{V_1},\mu|_{V_1} \rangle^* , \\
	\langle \mu_H|v \rangle - \langle \alpha | [v_H,v] \rangle
		&= 0 .
\end{cases}
\]
Solving these equations in $v_H$ and $\mu_H$,
we obtain that 
\begin{equation}\label{eq63d398f6}
	\begin{cases}
	v_H(g,\alpha) &= (\alpha|_{V_1})^\sharp \in V_1 , \\
	\mu_H(g,\alpha) &= \alpha\circ\ad_{(\alpha|_{V_1})^\sharp}  .
	\end{cases}
\end{equation}
Therefore, the Hamiltonian flow is given by curves $(g(t),\alpha(t))$ solving~\eqref{eq63d3a6ab}, that is,
\begin{equation}\label{eq63d39dbe}
	\begin{cases}
	DL_{g(t)}^{-1}\dot g(t) &= (\alpha(t)|_{V_1})^\sharp ,\\
	\dot \alpha(t) &= \alpha(t)\circ\ad_{(\alpha(t)|_{V_1})^\sharp} .
	\end{cases}
\end{equation}

\begin{proposition}\label{prop63d45030}
	Let $(g,\alpha):I\to G\times\frk g^*$ be a solution to~\eqref{eq63d39dbe} with $g(0)=e$, the identity element of $G$.
	Then $\alpha(t)$ is the restriction to the curve $g$ of a right invariant 1-form.
	In other words, for all $t\in I$,
	\[
	\alpha(0) 
	= \alpha(t)\circ \Ad_{g(t)^{-1}} .
	\]
\end{proposition}
\begin{proof}
	We show that the derivative in $t$ of $ \alpha(t)\Ad_{g(t)^{-1}}$ is zero for all $x\in\frk g$.
	So, we first see that
	\begin{align*}
	\left.\frac{\dd}{\dd\epsilon}\right|_{\epsilon = 0}
		& \alpha(t+\epsilon)\Ad_{g(t+\epsilon)^{-1}}  \\
	&= \left.\frac{\dd}{\dd\epsilon}\right|_{\epsilon = 0}
		 \alpha(t+\epsilon)\Ad_{g(t)^{-1}}
	+ \left.\frac{\dd}{\dd\epsilon}\right|_{\epsilon = 0}
		 \alpha(t)\Ad_{g(t+\epsilon)^{-1}} \\
	&= \alpha(t)\circ\ad_{(\alpha(t)|_{V_1})^\sharp}\Ad_{g(t)^{-1}} 
	+ \left.\frac{\dd}{\dd\epsilon}\right|_{\epsilon = 0}
		 \alpha(t)\Ad_{g(t+\epsilon)^{-1}g(t)}\Ad_{g(t)^{-1}} .
	\end{align*}
	Since 
	\begin{multline*}
	\left.\frac{\dd}{\dd\epsilon}\right|_{\epsilon = 0} g(t+\epsilon)^{-1}g(t) 
	=\left.\frac{\dd}{\dd\epsilon}\right|_{\epsilon = 0} (g(t)^{-1}g(t+\epsilon))^{-1} \\
	= - \left.\frac{\dd}{\dd\epsilon}\right|_{\epsilon = 0} g(t)^{-1}g(t+\epsilon)
	= - DL_{g(t)^{-1}}\dot g(t) 
	= - (\alpha(t)|_{V_1})^\sharp ,
	\end{multline*}
	we obtain
	\begin{multline*}
	\left.\frac{\dd}{\dd\epsilon}\right|_{\epsilon = 0}
		 \alpha(t)\Ad_{g(t+\epsilon)^{-1}g(t)}\Ad_{g(t)^{-1}}
	= \left.\frac{\dd}{\dd\epsilon}\right|_{\epsilon = 0}
		 \alpha(t)\Ad_{\exp(-\epsilon(\alpha(t)|_{V_1})^\sharp)}\Ad_{g(t)^{-1}} \\
	= - \alpha(t)\ad_{(\alpha(t)|_{V_1})^\sharp}\Ad_{g(t)^{-1}} .
	\end{multline*}
	We conclude that $\left.\frac{\dd}{\dd\epsilon}\right|_{\epsilon = 0}
		\alpha(t+\epsilon)\Ad_{g(t+\epsilon)^{-1}}=0$,
	and thus 
	\[
	\alpha(t)\Ad_{g(t)^{-1}} = \alpha(0)\Ad_{g(0)^{-1}} = \alpha(0).
	\]
\end{proof}

Proposition~\ref{prop63d45030} gives a reinterpretation of the ODE~\eqref{eq63d39dbe}.
Indeed, given $\alpha_0\in\frk g^*$, 
we first take the right-invariant 1-form $\alpha(g):=\Ad_{g}^*\alpha_0$, 
then we define the horizontal vector field $v_H(g) = (\alpha(g)|_{V_1})_\flat$,
and finally we integrate the vector field $\tilde v_H$ starting from $e$.
Explicitly, $\tilde v_H$ is
\[
\tilde v_H(g) = D L_g|_{e}((\Ad_{g}^*\alpha_0)|_{V_1})^\sharp .
\]

The ODE~\eqref{eq63d39dbe} has a few useful symmetries that we want to highlight.

\begin{lemma}[Symmetries of the sub-Riemannian Hamiltonian flow: change of speed]\label{lem63ef47ae}
	If $(g,\alpha):I\to G\times\frk g^*$ is a solution to~\eqref{eq63d39dbe},
	then $t\mapsto (g(\lambda t),\lambda\alpha(\lambda t))$ is also a solution to~\eqref{eq63d39dbe}, for every $\lambda>0$.
\end{lemma}
\begin{proof}
	Define $h(t) = g(\lambda t)$ and $\beta(t) = \lambda\alpha(\lambda t)$.
	Then 
	\[
	DL_{h(t)}^{-1}\dot h(t) 
	= \lambda DL_{g(\lambda t)}^{-1} \dot g(\lambda t)
	= \lambda (\alpha(\lambda t)|_{V_1})^\sharp
	= (\beta(t)|_{V_1})^\sharp ,
	\]
	and 
	\[
	\dot\beta(t)
	= \lambda^2 \dot\alpha(\lambda t)
	= \lambda^2 \alpha(\lambda t)\circ\ad_{(\alpha(\lambda t)|_{V_1})^\sharp}
	= \beta( t)\circ\ad_{(\beta( t)|_{V_1})^\sharp} .
	\]
	Therefore, $(h,\beta)$ is a solution to~\eqref{eq63d39dbe}.
\end{proof}

\begin{lemma}[Symmetries of the sub-Riemannian Hamiltonian flow: left translations]\label{lem6427c263}
	If $p\in G$ and if $(g,\alpha):I\to G\times\frk g^*$ is a solution to~\eqref{eq63d39dbe},
	then 
	\begin{equation*}
	t\mapsto \left( L_p(g(t)) , \alpha(t)  \right) 
	\end{equation*}
	is also a solution to~\eqref{eq63d39dbe}.
\end{lemma}
\begin{proof}
	The second equation in~\eqref{eq63d39dbe} does not depend on $g(t)$.
	In the first equation we have
	\[
	DL^{-1}_{pg(t)}(DL_p\dot g(t)) = DL^{-1}_{g(t)}\dot g(t) ,
	\]
	and thus the curve $(L_p(g(t)) , \alpha(t))$ is still a solution to~\eqref{eq63d39dbe}.
\end{proof}

\begin{lemma}[Symmetries of the sub-Riemannian Hamiltonian flow: homotheties]\label{lem63ef47a6}
	Let $L:G\to G$ be a Lie group automorphism with Lie algebra automorphism $\ell:\frk g\to\frk g$.
	Assume that $\ell(V_1)=V_1$ and that there exists $\lambda\in\R$ with $\langle \ell v,\ell w \rangle = \lambda^2 \langle v,w \rangle$ for all $v,w\in V_1$.
	
	If $(g,\alpha):I\to G\times\frk g^*$ is a solution to~\eqref{eq63d39dbe},
	then 
	\begin{equation}\label{eq63ef48b3}
	t\mapsto \left( L(g(t)) , \lambda^2 \alpha(t)\circ\ell^{-1}  \right) 
	\end{equation}
	is also a solution to~\eqref{eq63d39dbe}.
\end{lemma}
In fact, the existence of a homothety like in Lemma~\ref{lem63ef47a6} implies that the group is a Carnot group, see \cite{MR4256009}.
\begin{proof}
	First of all, notice that, if $\alpha\in V_1^*$, then for all $w\in V_1$ we have
	\[
	\langle \ell\alpha^\sharp,w \rangle
	= \lambda^2 \langle \alpha^\sharp , \ell^{-1} w \rangle
	= \lambda^2 \langle \alpha | \ell^{-1}w \rangle
	= \langle \lambda^2\alpha\circ\ell^{-1} | w \rangle .
	\]
	Therefore
	\begin{equation}\label{eq6492c851}
	\ell\alpha^\sharp
	= \lambda^2 \left(\alpha\circ\ell^{-1}\right)^\sharp .
	\end{equation}
	
	Next, define $h(t) = L(g(t))$ and $\beta(t) = \lambda^2 \alpha(t)\circ\ell^{-1} $.
	Then, using both~\eqref{eq6492c851} and~\eqref{eq63d39dbe}, we have
	\[
	DL_{h(t)}^{-1}\dot h(t)
	= \ell DL_{g(t)}^{-1}\dot g(t)
	= \ell (\alpha(t)|_{V_1})^\sharp
	= \left( \lambda^2\alpha (t)|_{V_1} \circ\ell^{-1}  \right)^\sharp
	= (\beta(t)|_{V_1})^\sharp .
	\]
	Similarly,
	\begin{align*}
	\dot\beta(t)
	&= \lambda^2  \dot\alpha(t)\circ\ell^{-1} 
	= \lambda^2 \alpha(t)\circ\ad_{(\alpha(t)|_{V_1})^\sharp}\circ\ell^{-1}  \\
	&= \lambda^2 \alpha(t)\circ\ell^{-1}\circ\ad_{\ell(\alpha(t)|_{V_1})^\sharp} \\
	&= \lambda^2 \alpha(t)\circ\ell^{-1}\circ\ad_{\left(\lambda^2\alpha(t)\circ\ell^{-1}|_{V_1}\right)^\sharp}  \\
	&= \beta(t)\circ\ad_{\left(\beta(t)|_{V_1}\right)^\sharp} .
	\end{align*}
	Therefore, $(h,\beta)$ is a solution to~\eqref{eq63d39dbe}.
\end{proof}

\subsection{The sub-Riemannian exponential map}
\label{sec64dc6371}
Notice that the Hamiltonian vector field $\cal X_H$ defined in~\eqref{eq63d1b115} is complete.
Indeed, by Proposition~\ref{prop63d45030}, we only need to show that if $(g,\alpha):(a,b)\to G\times\frk g^*$ is an integral curve of $\cal X_H$, then $g:(a,b)\to G$ can be continuously extended to the closed interval $[a,b]$.
We know that the curve $g$ is a length-minimizing curve parametrized by constant speed
with respect to a left-invariant sub-Riemannian distance on $G$, see for instance~\cite{MR1867362,MR3971262}.
Since such a distance is complete, the curve $g$ has a continuous extension to the closed interval $[a,b]$.
It follows that $\cal X_H$ is a complete vector field.

\begin{definition}
	Given a Lie group $G$ with Lie algebra $\frk g$ and a bracket generating subspace $V_1\subset\frk g$ endowed with a scalar product $\langle \cdot,\cdot \rangle$.
	We consider the left-invariant Hamiltonian function 
	$H:G \times \frk g^*\to\R$ defined as in~\eqref{eq6412f2b5}.
	The \emph{sub-Riemannian exponential map} is the function 
	\[
	\sre:\frk g^*\to G
	\] 
	that maps every $\xi\in \frk g^* = T_e^*G$ to the end point $\sre(\xi) = g(1)$ of the solution $(g,\alpha):[0,1]\to T^*G$ of the Hamiltonian system~\eqref{eq63d3a6ab}, 
	which becomes in this case~\eqref{eq63d39dbe},
	with $g(0)=e$ and $\alpha(0)=\xi$.
\end{definition}

Notice that $\sre$ is an analytic function defined on the whole space $\frk g^*$.
Indeed, since the Hamiltonian $H$ is analytic, the Hamiltonian vector field $\cal X_H$ defined in~\eqref{eq63d1b115} is also analytic.
By the Cauchy--Kovalevskaya Theorem, the flow of $\cal X_H$ on $T^*G$ is analytic.
Since $\sre$ is the restriction of the flow of $\cal X_H$ to $T_e^*G$ composed with the bundle projection $T^*G\to G$, we conclude that $\sre$ is an analytic function.

\section{Preliminaries on Carnot groups}
\label{sec6492eb72}

\subsection{Carnot groups}
For a more detailed introduction to Carnot groups, we suggest to consult~\cite{MR3742567}.
A \emph{Carnot group} is a connected simply connected Lie group $G$ whose Lie algebra $\frk g$ 
has a fixed stratification $\frk g = \bigoplus_{j=1}^s V_j$
and a fixed scalar product $\langle \cdot,\cdot \rangle$ on $V_1$.
A \emph{stratification} is a linear splitting $\frk g = \bigoplus_{j=1}^s V_j$ where $[V_1,V_j]=V_{j+1}$ for $j\in\{1,\dots,s-1\}$ and $[V_1,V_s]=\{0\}$.

Since $V_1$ Lie generates $\frk g$, 
the left-invariant horizontal vector bundle $\tilde V_1\subset TG$ is bracket generating
and together with the scalar product $\langle \cdot,\cdot \rangle$ on $V_1$,
a sub-Riemannian distance is determined on $G$, see~\cite{MR3742567}.
Although we will not directly deal with distances in this article,
when we will speak of length-minimizing curves we mean with respect to the sub-Riemannian distance induced by the choice of $V_1$ and of $\langle \cdot,\cdot \rangle$ on $V_1$.

Since $G$ turns out to be a nilpotent group,
the group exponential map $\exp:\frk g\to G$ is a global diffeomorphism. Furthermore, the Haar measure $\vol$ on $G$ is just the pushforward measure of the Lebesgue measure on $\frk g$ by $\exp$.

We assume $V_s\neq\{0\}$, and so $G$ has \emph{step} $s$.
The \emph{(topological) dimension} of $G$ is $n = \sum_{j=1}^s \dim(V_j)$;
the \emph{homogeneous dimension} of $G$ is $Q =\sum_{j=1}^s j\dim(V_j)$.

The dual space $\frk g^*$ inherits a splitting $\frk g^* = \bigoplus_{j=1}^sV_j^*$, 
where $V_j^* = \{\alpha\in\frk g^*:V_i\subset\ker\alpha,\ \forall i\neq j\}$.

For $\lambda\in\R$, \emph{dilation of factor $\lambda$} on $G$ is the group automorphism 
$\delta_\lambda:G\to G$ 
whose induced Lie algebra automorphism $(\delta_\lambda)_*:\frk g\to\frk g$ 
is the linear map 
$(\delta_\lambda)_*v = \lambda^jv$ for $v\in V_j$.
We usually denote $(\delta_\lambda)_*$ again by $\delta_\lambda$.

\subsection{Symmetries of the sub-Riemannian Exponential map on Carnot groups}
\label{sec6425df39}

Lemma~\ref{lem63ef47a6} translates to symmetries of the sub-Riemannian exponential map on Carnot groups.
For $\lambda\in\R\setminus\{0\}$, define $\eta_\lambda:\frk g^*\to\frk g^*$ as
\[
\eta_\lambda\left( \sum_{j=1}^s \xi_j \right) := \sum_{j=1}^s \lambda^{2-j} \xi_j .
\]
From Lemma~\ref{lem63ef47a6}, we obtain, for all $\xi\in\frk g^*$ and $\lambda\in\R\setminus\{0\}$,
\begin{equation}\label{eq64009645}
	\sre(\eta_\lambda \xi) 
	=
	\delta_\lambda\sre(\xi) 
	.
\end{equation}
Taking the Jacobian determinant in~\eqref{eq64009645}, we also get for $\lambda\in\R\setminus\{0\}$
\begin{equation}\label{eq64118940}
	\lambda^{2n-Q} \Jac(\sre)(\eta_\lambda(\xi))
	=
	\lambda^Q \Jac(\sre)(\xi), 
\end{equation}
where we used the fact that 
$\det(\eta_\lambda) = \lambda^{2n}\det(\delta_{1/\lambda}) = \lambda^{2n-Q}$.

\begin{remark}\label{rJac}
For the precise meaning of the Jacobian determinant here, we first fix coordinates on $\frk g^*$ and $\frk g$ (thus on $G$ by the group exponential map $\exp$) which preserve the Carnot group stratification respectively. Then the Jacobian determinant can be calculated in the usual Euclidean sense. Although different choices of the coordinates will change the value of the Jacobian determinant by multiplying a nonzero constant, it turns out that our definitions of $\Kmin(\xi)$, $\Kmin(G)$ and $\Kmax(G)$ in \eqref{eq6400a110} are independent of the choice of the coordinates and our proofs below remain the same. Furthermore, noticing by definition the geodesic dimension $N_{GEO}$ and the curvature exponent $N_{CE}$ are  
independent of the choice of the Haar measure,	in the following we can assume the Haar measure $\vol$ on $G$ is exactly the pushforward measure of the Lebesgue measure induced by the fixed coordinates on $\frk g$ without loss of generality. 
\end{remark}

For $\lambda\in\R$, define $\zeta_\lambda:\frk g^*\to\frk g^*$ as
\[
\zeta_\lambda\left( \sum_{j=1}^s \xi_j \right) := \sum_{j=1}^s \lambda^{j-1} \xi_j .
\]
Then $\lambda\xi = \eta_\lambda\zeta_\lambda(\xi)$ for all $\lambda\in\R\setminus\{0\}$.
If we substitute $\xi$ with $\zeta_\lambda(\xi)$ in~\eqref{eq64118940},
we get
\begin{equation}\label{eq64009980}
	\Jac(\sre)(\lambda\xi) = \lambda^{2Q-2n} \Jac(\sre)(\zeta_\lambda(\xi)) ,
\end{equation}
for all $\lambda\in\R$.
Notice that, when $\lambda\to0$, we have $\zeta_\lambda(\xi)\to\xi_1\in V_1^*$.

Recall Definition~\ref{def6401a951}: 
$\scr D\subset\frk g^*$ is the open set of all $\xi\in\frk g^*$ such that
$\Jac(\sre)(t\xi)\neq0$ for all $t\in(0,1]$ and $t\mapsto\sre(t\xi)$ is the unique constant-speed length-minimizing curve $[0,1]\to G$ from $e$ to $\sre(\xi)$.
From this definition, it follows that if $\xi\in\scr D$ then $t\xi\in\scr D$ for all $t\in(0,1]$.

We have defined $\cal S_e\subset G$ as the image $\sre(\scr D)$, 
so that $\sre$ is a diffeomorphism from $\scr D$ to $\cal S_e$.
We denote by $\rho:\cal S_e\to\scr D$ the inverse of $\sre$ on $\cal S_e$.
It is well known that $\cal S_e$ is dense in $G$, 
although it is not known whether it has full measure.

\begin{lemma}\label{lem64268086}
	If $\xi\in\scr D$, then $\zeta_\lambda(\xi)\in\scr D$ for all $\lambda\in(0,1]$.
\end{lemma}
\begin{proof}
	We will apply 
	the Definition~\ref{def6401a951} itself.
	Fix $\xi\in\scr D$ and $\lambda\in(0,1]$.
	
	Firstly, from~\eqref{eq64009980}, we have for all $t\in(0,1]$,
	\[
	\Jac(\sre)(t\zeta_\lambda(\xi)) 
	= \Jac(\sre)(\zeta_\lambda(t \xi))
	= \frac{ \Jac(\sre)(\lambda t\xi) }{ \lambda^{2Q-2n} }
	\neq0
	\]
	because $0<\lambda t\le 1$.
	
	Secondly, suppose that $\eta:[0,1]\to G$ is a constant-speed length-minimizing curve from $e$ to $\sre(\zeta_\lambda(\xi))$.
	Then $\gamma(t):=\delta_\lambda(\eta(t))$ is also a constant-speed length-minimizing curve from $e$ to 
	\[
	\delta_\lambda(\sre(\zeta_\lambda(\xi))) 
	= \sre(\eta_\lambda\zeta_\lambda(\xi)) 
	= \sre(\lambda\xi) ,
	\]
	where we have used~\eqref{eq64009645}.
	Since $\lambda\xi\in\scr D$, then $\gamma(t) = \sre(t\lambda\xi)$.
	Therefore,
	\[
	\eta(t) = \delta_{1/\lambda}\gamma(t)
	= \delta_{1/\lambda}\sre(t\lambda\xi)
	= \sre(\eta_{1/\lambda}\eta_\lambda\zeta_\lambda(t\xi))
	= \sre(t\zeta_\lambda(\xi)) .
	\]
	We have thus shown that $\zeta_\lambda(\xi)\in\scr D$.
\end{proof}

Since $\sre$ is an analytic map, the function $\frk g^*\times \R \to\R$, $(\xi,\lambda)\mapsto \Jac(\sre)(\zeta_\lambda(\xi))$, is also analytic.
In particular,
there are analytic functions $P_k:\frk g^*\to\R$
such that
for every $U\subset\frk g^*$ bounded there is $\lambda_0>0$ such that for all $\xi\in U$ and $\lambda\in(-\lambda_0,\lambda_0)$,

\begin{equation}\label{eq6492e37b}
\Jac(\sre)(\zeta_\lambda(\xi))
= \sum_{k=0}^\infty P_k(\xi) \lambda^k ,
\end{equation}
where the series converges absolutely and uniformly in $\xi\in U$. 

\begin{lemma}\label{lem6426f63f}
	Each $P_k$ is a homogeneous polynomial of degree $2Q-2n+k$, 
	that is, for every $\xi\in\frk g^*$ and $\mu\in\R$,
	\begin{equation}\label{eq6426f407}
		P_k(\mu\xi) = \mu^{2Q-2n+k} P_k(\xi).
	\end{equation}
	In particular, we have 
	\begin{equation}\label{homoG}
	\Kmin(\mu \xi) = \Kmin(\xi) , \quad \forall \mu \ne 0, \xi \in \frk g^*.
	\end{equation}
\end{lemma}
\begin{proof}
	Using the same notation we had for~\eqref{eq6492e37b}, 
	together with the linearity of the maps $\zeta_\lambda$ and the group property $\zeta_\mu\zeta_\lambda = \zeta_{\mu\lambda}$,
	we obtain for $\lambda$ and $\mu$ small enough,	
	\begin{align*}
		\sum_{k=0}^\infty P_k(\mu\xi) \lambda^k
		&= \Jac(\sre)(\mu\zeta_\lambda(\xi)) 
		= \mu^{2Q-2n} \Jac(\sre)(\zeta_\mu\zeta_\lambda(\xi)) \\
		&= \mu^{2Q-2n} \Jac(\sre)(\zeta_{\mu\lambda}(\xi)) 
		= \sum_{k=0}^\infty \mu^{2Q-2n+k} P_k(\xi) \lambda^k .
	\end{align*}
	Analyticity implies that~\eqref{eq6426f407} holds
	for all $\xi\in\frk g^*$, all $\mu\in\R$ and all $k\in\N$.
\end{proof}

\subsection{Intermediate points of a negligible set are negligible}

The goal of this section is to show the following Proposition~\ref{prop64c38067} after two auxiliary lemmas.
We will then have two consequences.
First, the equivalence of $\vol(Z_\lambda(e,E))$ with standard definitions in the literature, see Remark~\ref{rem64c3884d}.
Next, the formula~\eqref{eq6425debc} in Proposition~\ref{prop64255e1c} for the volume of $Z_\lambda(e,E)$.

\begin{proposition}\label{prop64c38067}
	Let $G$ be a Sard-regular Carnot group.
	If $E\subset G$ has measure zero, then $Z_t(e,E)$ has also measure zero, for all $t\in[0,1]$.
\end{proposition}

\begin{remark}\label{rem64c3884d}
	We obtain from Proposition~\ref{prop64c38067} that, 
	in Sard-regular Carnot groups,
	our definition of intermediate points~\eqref{eq64c38a13}
	is ``almost equivalent'' to the definition given in~\cite{MR3852258,MR3502622}.
	
	Indeed, the set $E_{e,\lambda}$ defined in 
	\cite[Definition 5.43]{MR3852258}
	or
	\cite[Eq.(1)]{MR3502622}
	corresponds, in our notation, to $Z_\lambda(e,E\cap\cal S_e)$.
	However, Proposition~\ref{prop64c38067} implies that 
	$\vol(Z_\lambda(e,E)) = \vol(Z_\lambda(e,E\cap\cal S_e)) = \vol(E_{e,\lambda})$.
\end{remark}

The following lemma is well known to experts.
For the notions of \emph{End point map}, \emph{regular} and \emph{strictly normal} curves, 
see \cite{MR3308395,MR3971262}

\begin{lemma}\label{lem64c3701a}
	Let $\gamma:[0,1]\to G$ be a length-minimizing curve parametrized with constant speed.
	Suppose that there exist $\xi\in\scr D$ and $t$ such that $\gamma(ut)=\sre(u\xi)$ for all $u\in[0,1]$.
	Then there exists $\eta\in\frk g^*$ such that $\gamma(u)=\sre(u\eta)$ for all $u\in[0,1]$.
\end{lemma}
\begin{proof}
	Since $\xi\in\scr D$, then the restriction $\gamma|_{[0,t]}$ is \emph{regular} for the End point map, because the image of the differential of the End point map contains the image of the differential of the sub-Riemannian exponential map (see the proof of Lemma 2.31 in \cite{MR3569245}).
	It follows that $\gamma$ is also regular for the End point map, and thus \emph{strictly normal
}.
	In particular, there exists $\eta\in\frk g^*$ such that $\gamma(u)=\sre(u\eta)$ for all $u\in[0,1]$.
\end{proof}

\begin{lemma}\label{lem64c37bb5}
	If $E\subset G$ has zero measure, then $\sre^{-1}(E)\subset\frk g^*$ has also zero measure.
\end{lemma}
\begin{proof}
	Suppose that $\sre^{-1}(E)\subset\frk g^*$ has positive measure.
	Since $\Jac(\sre)$ is an analytic function, there is a Lebesgue's density point $\xi$ of $\sre^{-1}(E)$ (see \cite[Theorem~1.8]{MR1800917}) such that $\Jac(\sre)(\xi)\neq0$.
	Hence, $\sre$ is a diffeomorphism on a neighbourhood of $\xi$ and thus $\sre(\sre^{-1}(E))$ has positive measure.
	Since $\sre(\sre^{-1}(E)) \subset E$, then $E$ has positive measure too.
\end{proof}

\begin{proof}[Proof of Proposition~\ref{prop64c38067}]
	Set $\Sing=G\setminus\cal S_e$.
	We decompose $Z_t(e,E)$ into the following three sets:
	\[
	Z_t(e,E) = Z_t(e,E\cap\cal S_e) \cup (Z_t(e,E\cap\Sing)\cap\Sing) \cup (Z_t(e,E\cap\Sing)\cap\cal S_e)
	\]
	
	Case 1: volume of $Z_t(e,E\cap\cal S_e)$.
	In this case we have
	\[
	Z_t(e,E\cap\cal S_e) = \sre(t \rho(E\cap\cal S_e))
	\]
	where $\vol(E\cap\cal S_e)=0$ and where $p\mapsto \sre(t \rho(p))$ is a diffeomorphism in a neighbourhood of $E\cap\cal S_e$.
	Since $\vol(E)=0$, then $\vol(Z_t(e,E\cap\cal S_e))=0$.

	Case 2: volume of $Z_t(e,E\cap\Sing)\cap\Sing$.
	Since we assume $G$ to be Sard-regular, then $\vol(\Sing)=0$ 
	and thus $\vol(Z_t(e,E\cap\Sing)\cap\Sing)=0$.
	
	Case 3: volume of $Z_t(e,E\cap\Sing)\cap\cal S_e$.
	We claim that 
	\begin{equation}\label{eq64c381b7}
	Z_t(e,E\cap\Sing)\cap\cal S_e \subset \sre(t \sre^{-1}(E)) .
	\end{equation}
	Indeed, let $z\in Z_t(e,E\cap\Sing)\cap\cal S_e$.
	Then there are a point $q\in E\cap\Sing$ and a length-minimizing geodesic $\gamma:[0,1]\to G$ parametrized with constant speed with $\gamma(0)=e$, $\gamma(1)=q$ and $\gamma(t)=z$.
	Since $z\in\cal S_e$, there is $\xi\in\scr D$ such that $\sre(\xi)=z$.
	By the definition of $\scr D$, we have $\sre(u\xi) = \gamma(ut)$ for all $u\in[0,1]$.
	From Lemma~\ref{lem64c3701a}, it follows that there is $\eta\in\frk g^*$ such that $\sre(u\eta)=\gamma(u)$.
	Thus, the claim~\eqref{eq64c381b7} is proven.
	
	By Lemma~\ref{lem64c37bb5}, we have $\vol(\sre(t \sre^{-1}(E))) = 0$ 
	and thus~\eqref{eq64c381b7} implies that $\vol(Z_t(e,E\cap\Sing)\cap\cal S_e) = 0 $.
\end{proof}

\begin{proposition}\label{prop64255e1c}
	If $G$ is a Sard-regular Carnot group, then,
	for every measurable $E\subset G$,
	\begin{equation}\label{eq6425debc}
		\vol(Z_\lambda(e,E)) = \lambda^{2Q-n}  \int_{\rho(E \cap \cal S_e)} | \Jac(\sre)(\zeta_\lambda(\xi)) | \dd\xi .
	\end{equation}
\end{proposition}
\begin{proof}
From Definition~\ref{def6401a951} 
and the definition of intermediate points,
we get 
that
\begin{align*}
Z_\lambda(e,E) 
&= Z_\lambda(e,E\cap\cal S_e) \cup Z_\lambda(e,E\setminus\cal S_e) \\
&= \sre(\lambda\rho(E \cap \cal S_e)) \cup Z_\lambda(e,E\setminus\cal S_e) .
\end{align*}
Proposition~\ref{prop64c38067} says that $\vol(Z_\lambda(e,E\setminus\cal S_e))=0$.

	Thus, by Remark \ref{rJac}, the area formula and the identity~\eqref{eq64009980},
	we conclude
	\begin{align*}
	\vol(Z_\lambda(e,E))
	&= \vol(\sre(\lambda\rho(E \cap \cal S_e))) \\
	&= \int_{\lambda\rho(E \cap \cal S_e )} |\Jac(\sre)(\xi)| \dd \xi \\
	&= \lambda^n \int_{\rho(E \cap \cal S_e)} |\Jac(\sre)(\lambda\xi) |\dd \xi \\
	&= \lambda^{2Q-n} \int_{\rho(E \cap \cal S_e)} |\Jac(\sre)(\zeta_\lambda(\xi))| \dd\xi .
	\end{align*} 
\end{proof}

\section{Proof of Theorem~\ref{thm6401aa05}}
\label{sec6425dfb4}
 
Thanks to Remark~\ref{rem64c3884d}, we can apply the following result 
by Agrachev--Barilari--Rizzi.

\begin{proposition}[{\cite[Theorem D, page 58]{MR3852258}}]\label{prop64c384b7}
	For any bounded, measurable set $E\subset\cal S_e$ with $0<\vol(E)<+\infty$ we have
	$\vol(Z_\epsilon(e,E))\sim \epsilon^{N_{GEO}}$
	for $\epsilon\to 0$.
\end{proposition}

\begin{proof}[Proof of Theorem~\ref{thm6401aa05}]
	We will give a measurable set $E\subset G$ with $0<\vol(E)<\infty$ such that 
	\begin{equation}\label{eq64958508}
	\lim_{\epsilon\to0}
		\frac{ \vol(Z_\epsilon (e,E)) }{ \epsilon^{2Q-n+\Kmin(G)} }
		\in (0,+\infty) ,
	\end{equation}
	and then we will apply Proposition~\ref{prop64c384b7} (i.e., \cite[Theorem D, page 58]{MR3852258})
	to conclude that $N_{GEO}=2Q-n+\Kmin(G)$.
	
	Since $\scr E := \{\xi : P_{\Gamma(G)}(\xi) \ne 0\}$ is open and dense in $\frk g^*$, the set $\scr D \cap \scr E$ is open and non-empty. 
	As a result, 
	we can choose $\xi_0 \in \scr D \cap \scr E$ and 
	a compact connected neighbourhood $U \subset \scr D \cap \scr E$ of $\xi_0$
	and an open interval $I\subset\R$ with $0\in I$
	such that for $\epsilon\in I$ and $\xi\in U$ we have
	\[
	\Jac(\sre)(\zeta_\epsilon(\xi))
	= \sum_{k = \Gamma(G)}^\infty P_k(\xi) \epsilon^k ,
	\]
	where $P_k:U\to\R$ are analytic functions,
	the series is absolutely convergent
	 and
	$P_{\Gamma(G)}(\xi)\neq0$ 
	for all $\xi\in U$.
	In particular, if $P_{\Gamma(G)}(\xi_0)>0$ ($P_{\Gamma(G)}(\xi_0)<0$, resp.), 
	then there is $\eta>0$ such that 
	$P_{\Gamma(G)}(\xi)>\eta$ ($P_{\Gamma(G)}(\xi)<-\eta$, resp.) 
	for all $\xi\in U$.
	
	Let assume $P_{\Gamma(G)}(\xi_0)>0$, as the other case is similar.
	Since $U$ is compact and since $(\epsilon,\xi)\mapsto\frac{ \Jac(\sre)(\zeta_\epsilon(\xi)) }{ \epsilon^{\Gamma(G)} }$
	is analytic in a neighbourhood of $\{0\}\times U$, 
	the limit
	\[
		\lim_{\epsilon\to0} \frac{ \Jac(\sre)(\zeta_\epsilon(\xi)) }{ \epsilon^{\Gamma(G)} }
		= P_{\Gamma(G)}(\xi)
	\]
	is uniform for $\xi\in U$.
	Therefore, there exists $\epsilon_0>0$ such that 
	\[
		\Jac(\sre)(\zeta_\epsilon(\xi)) > \frac{\eta}{2} \epsilon^{\Gamma(G)} > 0
	\]
	for all $\epsilon\in(0,\epsilon_0)$.
	
	Set $E := \sre(U)$.
	For $\epsilon\in(0,\epsilon_0)$,
	we obtain from Proposition~\ref{prop64255e1c}
	\begin{align*}
	\vol(Z_\epsilon(e,E))
		&= \epsilon^{2Q-n} \int_{U} |\Jac(\sre)(\zeta_\epsilon(\xi))| \dd\xi \\
		&= \epsilon^{2Q-n} \int_{U} \Jac(\sre)(\zeta_\epsilon(\xi)) \dd\xi \\
		&= \epsilon^{2Q-n+\Gamma(G)}  \sum_{k = \Gamma(G)}^\infty \epsilon^{k - \Gamma(G)} \int_{U} P_k(\xi) \dd\xi.
	\end{align*}
	Since $\int_{U} P_{\Gamma(G)}(\xi) \dd\xi > 0$ from our choice of $U$, 
	we have
	\[
	\lim_{\epsilon\to0}
		\frac{ \vol(Z_\epsilon (e,E)) }{ \epsilon^{2Q-n+\Kmin(G)} }
	= \int_{U} P_{\Gamma(G)}(\xi) \dd\xi 
	\in (0,\infty) ,
	\]
	that is,~\eqref{eq64958508}. 
\end{proof}

\section{Proof of Theorem~\ref{thm6419896b}}
\label{sec6425e252}

\begin{proposition}\label{prop6401afb9}
	On a Sard-regular Carnot group, the following statements are equivalent for every $N>0$:
	\begin{enumerate}[label=(\roman*)]
	\item\label{item64267a2a_1}
	$N\ge N_{CE}$;
	\item\label{item64267a2a_2}
	$\lambda^{2Q-n} |\Jac(\sre)(\zeta_\lambda(\xi))| \ge \lambda^N |\Jac(\sre)(\xi)|$
	for all $\xi\in\scr D$ and $\lambda\in[0,1]$;
	\item\label{item64267a2a_3}
	$\left.\frac{\dd}{\dd \lambda}\right|_{\lambda=1} \left[ \frac{\Jac(\sre)(\zeta_\lambda(\xi)) }{ \lambda^{N-2Q+n} } \right]^2 \le 0$
	for all $\xi\in\scr D$;
	\item\label{item64267a2a_4}
	$\left.\frac{\dd}{\dd \lambda}\right|_{\lambda=\lambda_0} \left[  \frac{\Jac(\sre)(\zeta_\lambda(\xi)) }{ \lambda^{N-2Q+n} } \right]^2 \le 0$
	for all $\xi\in\scr D$ and $\lambda_0\in(0,1]$.
	\end{enumerate}
\end{proposition}
\begin{proof}
	\ref{item64267a2a_1}$\THEN$\ref{item64267a2a_2}
	If $N\ge N_{CE}$, then 
	\[
	\vol(Z_\lambda(p,E)) \ge \lambda^N \vol(E) 
	\qquad\forall \lambda\in[0,1] ,
	\]
	for all measurable $E\subset G$ with $0<\vol(E)<\infty$.
	By Proposition~\ref{prop64255e1c}, we then have, for every $U\subset\scr D$ open and bounded,
	\[
	\lambda^{2Q-n}  \int_{U} |\Jac(\sre)(\zeta_\lambda(\xi))| \dd\xi
	\ge \lambda^N  \int_{U} |\Jac(\sre)(\xi)| \dd\xi ,
	\qquad\forall \lambda\in[0,1] .
	\]
	By the Lebesgue differentiation theorem, 
	and by the continuity in $\xi$ of both left- and right-hand integrand functions,
	\ref{item64267a2a_2} follows.
	
	\ref{item64267a2a_2}$\THEN$\ref{item64267a2a_1}
	By a direct application of Proposition~\ref{prop64255e1c},
	the pointwise estimate for $\Jac(\sre)$ implies the volume estimate of 
	the definition of $N_{CE}$.
	
	\ref{item64267a2a_2}$\THEN$\ref{item64267a2a_3}
	The inequality~\ref{item64267a2a_2} implies that the function
	\[
	f_\xi(\lambda) := \frac{\Jac(\sre)(\zeta_\lambda(\xi)) }{ \lambda^{N-2Q+n} }
	\]
	satisfies $f^2_\xi(\lambda)\ge f^2_\xi(1)$ for $\lambda\in[0,1]$.
	Since $f_\xi$ is smooth, we conclude $\left.\frac{\dd}{\dd \lambda}\right|_{\lambda=1} f_\xi^2(\lambda) \le 0$.
	
	\ref{item64267a2a_3}$\THEN$\ref{item64267a2a_4}
	Fix $\xi\in\scr D$ and $\lambda_0\in(0,1]$.
	Notice that, since $\zeta_{\lambda_0\lambda} = \zeta_{\lambda_0}\circ\zeta_\lambda$,
	the function $f_\xi(\lambda)$ defined above satisfies
	\[
	f_{\zeta_{\lambda_0}\xi}(\lambda) = \lambda_0^{N-2Q+n} f_\xi(\lambda_0\lambda)
	\]
	for all $\lambda\in(0,1]$.
	Moreover, by Lemma~\ref{lem64268086}, we have $\zeta_{\lambda_0}\xi\in\scr D$.
	Therefore, from~\ref{item64267a2a_3} we obtain
	\[
	0 
	\ge \left.\frac{\dd}{\dd \lambda}\right|_{\lambda=1} f_{\zeta_{\lambda_0}\xi}^2(\lambda)
	= \lambda_0^{2N-4Q+2n} \left.\frac{\dd}{\dd \lambda}\right|_{\lambda=1} f_\xi^2(\lambda_0\lambda)
	= \lambda_0^{2N-4Q+2n+1} \left.\frac{\dd}{\dd \lambda}\right|_{\lambda=\lambda_0} f_\xi^2(\lambda) .
	\]
	We conclude that~\ref{item64267a2a_4} holds.
	
	\ref{item64267a2a_4}$\THEN$\ref{item64267a2a_2}
	The hypothesis~\ref{item64267a2a_4} implies that $f_\xi^2$ is non increasing on $(0,1]$,
	whenever $\xi\in\scr D$. 
	It follows that  $f_\xi^2(\lambda) \ge f_\xi^2(1)$ for all $\xi\in\scr D$ and $\lambda\in(0,1]$,
	which is equivalent to~\ref{item64267a2a_2}.
\end{proof}

\begin{lemma} \label{supD}
On Sard-regular Carnot group $G$, we have 
\[
	\Kmax(G) = \sup\{\Kmin(\xi) : \xi\in\scr D\} .
\]
\end{lemma}

\begin{proof}
Let us denote the number on the RHS by $K$. It follows from the definition of $\scr D$ that if $\xi \in \scr D$, then $\Kmin(\xi) < \infty$. Then it follows from the original definition of $\Kmax(G)$ (cf.~\eqref{eq6400a110})
that $K \le \Kmax(G)$. 

To prove the converse inequality, by \eqref{homoG}, it suffices to prove that if $\xi \in \frk g^*$ with $\Kmin(\xi) < \infty$, then for $\mu > 0$ small enough we have $\mu \xi \in \scr D$.
We check Definition~\ref{def6401a951} for $\mu \xi$ with $\mu$ small enough. Since $\Kmin(\xi) < \infty$, for $\mu$ small and $t \in (0,1]$, we have 
\[
\Jac(\sre)(\zeta_{t\mu}(\xi))
= \sum_{k=\Kmin(\xi)}^\infty P_k(\xi) (t\mu)^k  \ne 0.
\]
It follows from \eqref{eq64009980} that
\[
	\Jac(\sre)(t\mu\xi) =(t\mu)^{2Q-2n} \Jac(\sre)(\zeta_{t\mu}(\xi)) \ne 0, \quad \forall t \in (0,1].
\]
By the local minimality (cf. \cite[Theorem 4.65]{MR3971262}), for $\mu$ small enough, $t\mapsto\sre(t\mu\xi)$ is the unique constant-speed length-minimizing curve $[0,1]\to G$ from $e$ to $\sre(\mu\xi)$. 
Thus we have proven that $\mu \xi \in \scr D$ and this ends the proof of the lemma.
\end{proof}

\begin{proof}[Proof of Theorem~\ref{thm6419896b}]
	Fix $\xi\in\scr D$ and assume
	\begin{equation}\label{eq6426ef35}
		N < 2Q-n+\Gamma(\xi) .
	\end{equation}
	For $\lambda\in(0,1]$, define the analytic function
	\[
	f_\xi(\lambda) := \frac{\Jac(\sre)(\zeta_\lambda(\xi)) }{ \lambda^{N-2Q+n} } .
	\]
	By~\eqref{eq64255fe9}, there is $\lambda_\xi>0$ such that, for $\lambda\in[0,\lambda_\xi)$, 
	\[
	f_\xi(\lambda)
	= \lambda^{\Gamma(\xi)-N+2Q-n} 
		\sum_{k\ge0} P_{\Gamma(\xi)+k}(\xi) \lambda^k.
	\]
	By standard rules of calculus, we have, for $\lambda\in[0,\lambda_\xi)$,
	\[
	f'_\xi(\lambda) = \lambda^{\Gamma(\xi)-N+2Q-n-1} 
		\sum_{k\ge0} (\Gamma(\xi)+k-N+2Q-n) P_{\Gamma(\xi)+k}(\xi) \lambda^k.
	\]
	Since $\Gamma(\xi)+k-N+2Q-n > 0$ by~\eqref{eq6426ef35}, as $\lambda\to0^+$, $f_\xi(\lambda)$ and $f_\xi'(\lambda)$ have the same sign as $P_{\Gamma(\xi)}(\xi)$.

	We conclude that 
	\[
	\frac{\dd}{\dd \lambda} f^2_\xi(\lambda) = 2 f_\xi(\lambda) f'_\xi(\lambda) > 0
	\]
	for $\lambda$ positive and small enough.
	Proposition~\ref{prop6401afb9} implies that $N<N_{CE}$.
	
	Since $N$ is arbitrary in $(0,2Q-n+\Gamma(\xi))$ and $\xi$ is arbitrary in $\scr D$, we complete the proof of Theorem~\ref{thm6419896b} by Lemma~\ref{supD}.
\end{proof}

\begin{remark}
From the definition in \eqref{eq6400a110} and Lemma \ref{lem6426f63f}, $\Kmin(\xi)$ attains its minimum value $\Kmin(G)$ on an open, non-empty Zariski subset of $\frk g^*$ (see also Remark \ref{rel1} and \cite[Proposition 5.46]{MR3852258}), which implies it is constant almost everywhere. It seems that the curvature exponent $N_{CE}$ should not detect higher values of $\Kmin(\xi)$ by definition. However, (ii) of Proposition \ref{prop6401afb9} shows that the curvature exponent $N_{CE}$ provides a uniform bound for the Jacobian determinant of the sub-Riemannian exponential map while $\Kmin(\xi)$ is defined in a pointwise way. Thus when $\xi$ moves in the set $\{\xi :  P_{\Kmin(G)}(\xi) \ne 0\}$, the coefficient of the first nonzero  term in \eqref{eq64255fe9}, or equivalently $P_{\Kmin(G)}(\xi)$, may become small and create an obstacle for the uniformity. This helps to explain why the curvature exponent $N_{CE}$ could be strictly larger than the geodesic dimension $N_{GEO}$.
\end{remark}

\section{Carnot groups of step two}
\label{sec6496dd08}

The following construction of the sub-Riemannian exponential map in Carnot groups of step two,
and in particular the splitting given in Definition~\ref{def63f121f2},
is linked to the techniques used in \cite{MR3513881,MR3671588} for the study of phase function of 
Fourier integral operators on Carnot groups. 
In fact, it has been explained to the first-named author by Alessio Martini.

\subsection{Preliminary observations on Carnot groups of step 2}
Let $\frk g = V_1\oplus V_2$ be a stratified Lie algebra of step 2,
with a scalar product $\langle \cdot,\cdot \rangle$ on $V_1$.

If $\mu\in V_2^*$, let $J_\mu:V_1\to V_1$ be the linear map 
defined by 
\[
\langle J_\mu v,w \rangle = \langle \mu | [v,w] \rangle
\]
for all $v,w\in V_1$.
Notice that $J_\mu$ is skew-symmetric, that is, 
\[
	J_\mu\in\so(V_1,\langle \cdot,\cdot \rangle) 
	= \{J:V_1\to V_1\text{ linear, }\langle Jx,y \rangle = - \langle x,Jy \rangle,\ \forall x,y\in V_1\} .
\]

We will fix a scalar product $\langle \cdot,\cdot \rangle$ on the whole $\frk g$ 
that extends the given one on $V_1$ and such that $V_1$ and $V_2$ are orthogonal.
We will use this scalar product to identify $\frk g$ with $\frk g^*$, via $v\mapsto v^\flat$.
For instance, if $u\in V_2$, we have $J_u:=J_{u^\flat}$.
It follows that, for all $v,w\in V_1$ and $u\in V_2$,
\begin{equation}\label{eq64c3d1c5}
	\langle u, [v,w] \rangle 
	= \langle u^\flat | [v,w] \rangle 
	= \langle J_{u^\flat} v,w \rangle  
	= \langle J_u v,w \rangle  .
\end{equation}

There is a special choice of scalar product on $V_2$:
we keep this result for completeness.

\begin{lemma}
	Denote by $\langle \cdot,\cdot \rangle_{HS}$ the Hilbert--Schmidt scalar product in $\so(V_1,\langle \cdot,\cdot \rangle)$, that is,
	\[
		\langle A,B \rangle_{HS} = \trace(AB^*) ,
	\]
	where $B^*$ is the conjugate of $B$ with respect to the scalar product $\langle \cdot,\cdot\rangle$ on $V_1$.

	There exists a unique extension of $\langle \cdot,\cdot \rangle$ from $V_1$ to a scalar product on $\frk g$ such that $V_1$ and $V_2$ are orthogonal and,
	for every $v,w\in V_2$, we have 
	\begin{equation}\label{eq63cfd3ea}
	\langle v,w \rangle = \langle J_{v^\flat},J_{w^\flat} \rangle_{HS} ,
	\end{equation}
	where $V_2\ni x\mapsto x^\flat\in V_2^*$ is the linear isomorphism induced by the scalar product.
\end{lemma}
\begin{proof}
	By definition and the fact $V_2=[V_1,V_1]$, 
	it is direct to check that the map $V_2^* \to \so(V_1,\langle \cdot,\cdot \rangle)$ defined by $\mu \mapsto J_\mu$ is linear and injective. 
	Then, there is a unique scalar product $\langle \cdot,\cdot \rangle^*$ on $V_2^*$ such that for every $\alpha,\beta \in V_2^*$,
	\begin{equation}\label{eq649442da}
	\langle \alpha,\beta \rangle^* = \langle J_\alpha, J_\beta \rangle_{HS}.
	\end{equation}
	The scalar product $\langle \cdot,\cdot \rangle^*$ on $V_2^*$ induces a scalar product $\langle \cdot,\cdot \rangle$ on $V_2$:
	combining this with the original $\langle \cdot,\cdot \rangle$ on $V_1$, 
	we obtain a scalar product on the whole $\frk g$ 
	such that $V_1$ and $V_2$ are orthogonal and~\eqref{eq63cfd3ea} holds by definition.
	
	Uniqueness is ensured because~\eqref{eq63cfd3ea}
	is equivalent to~\eqref{eq649442da}, which in turn
	uniquely determines the dual scalar product $\langle \cdot,\cdot \rangle^*$ on $V_2^*$, and thus on $V_2$.
\end{proof}

\subsection{The sub-Riemannian exponential map}

To integrate the first of the two equations in~\eqref{eq63d39dbe},
we will give the following construction of $G$:
First, we denote by $V=V_1\oplus V_2$ the vector space underlying the Lie algebra $\frk g$;
second, we define the Lie group $G$ as the smooth manifold $V$ endowed with the group operation
\[
a*b = a+b+\frac12[a,b] ,
\qquad a,b\in\frk g.
\]
It follows that $0$ is the identity element $e$, and $g^{-1}=-g$.
The advantage of this construction is that we will take derivatives as we do in the vector space $V$.
For instance, the differential at $0=e$ of the left translation $L_g:V\to V$ is the linear map $DL_g|_0:V\to V$,
\[
DL_g|_0[x] = x+\frac12[g,x] .
\]

\begin{lemma}\label{lem63d454ee}
	The ODE~\eqref{eq63d39dbe} is equivalent to
	\begin{equation}\label{eq63d455ac}
	\begin{cases}
		\dot x &= \xi , \\
		\dot u &= \frac12 [ x , \xi ] ,\\
		\dot \xi &= J_\mu \xi ,\\
		\dot \mu &= 0 ,
	\end{cases}
	\end{equation}
	for curves $((x,u),(\xi,\mu)) : I \to (V_1\oplus V_2)\oplus (V_1\oplus V_2)$.
	In other words, $((x,u),(\xi,\mu))$ is a solution to~\eqref{eq63d455ac} if and only if $((x,u),(\xi+\mu)^\flat)$ is solution to~\eqref{eq63d39dbe}.
\end{lemma}
\begin{proof}
	If $g=x+u\in G$ with $x\in V_1$ and $u\in V_2$, 
	and if $\alpha = (\xi+\mu)^\flat = \xi^\flat+\mu^\flat\in\frk g^*$ 
	with $\xi\in V_1$ and $\mu\in V_2$,
	then 
	$(\alpha|_{V_1})^\sharp = \xi$ and
	\[
	DL_g|_e[(\alpha|_{V_1})^\sharp] 
	= \xi + \frac12 [ x , \xi ] .
	\]
	Thus, we have that the first two equations in~\eqref{eq63d455ac} are equivalent to the first equation in~\eqref{eq63d39dbe}.
	
	Next, 
	if $v=v_1+v_2$ with $v_j\in V_j$, then
	\[
	\langle \alpha\circ\ad_{(\alpha|_{V_1})^\sharp} | v \rangle
	= \langle \alpha\circ\ad_\xi | v \rangle
	= \langle \alpha | [\xi,v] \rangle
	= \langle \mu,[\xi,v_1] \rangle
	= \langle J_\mu\xi, v_1 \rangle ,
	\]
	that is, $\alpha\circ\ad_{(\alpha|_{V_1})^\sharp} = (J_\mu\xi)^\flat$.	
	It follows that
	the second two equations in~\eqref{eq63d455ac} are equivalent to the second equation in~\eqref{eq63d39dbe}.
\end{proof}

\begin{proposition}\label{prop63ee3157}
	Given $\xi_0\in V_1$ and $\mu_0\in V_2$, the analytic solution 
	$((x,u),(\xi,\mu)) : \R \to (V_1\oplus V_2)\oplus (V_1\oplus V_2)$
	to~\eqref{eq63d455ac} with 
	$x(0)=0$, $u(0)=0$, $\xi(0)=\xi_0$ and $\mu(0)=\mu_0$ is
	\begin{align*}
	x(t) &= \frac{ e^{tJ_{\mu_0}}-\Id }{ J_{\mu_0} }\xi_0 , \\
	u(t) &= t^3 \sum_{k=1}^\infty B_k(\xi_0,\mu_0) t^{k-1} , \\
	\xi(t) &= e^{tJ_{\mu_0}}\xi_0 , \\
	\mu(t) &= \mu_0 ,
	\end{align*}
	where
	\begin{equation}\label{eq63ee2a7c}
	B_k(\xi_0,\mu_0) := \sum_{m=0}^k \frac{[J_{\mu_0}^m\xi_0,J_{\mu_0}^{k-m}\xi_0]}{2(m+1)!(k-m)!(k+2)} .
	\end{equation}
\end{proposition}
\begin{proof}
	Since $\dot\mu=0$, then $\mu(t) = \mu_0$.
	Since $\dot\xi = J_{\mu_0} \xi$, then $\xi(t) = e^{tJ_{\mu_0}}\xi_0$.
	Since $\dot x = \xi$ and $x(0)=0$, then
	\[
	x(t) = \frac{ e^{tJ_{\mu_0}}-\Id }{ J_{\mu_0} }\xi_0
	= \sum_{k=0}^\infty \frac{t^{k+1}}{(k+1)!} J_{\mu_0}^k\xi_0 .
	\]
	Finally, since $\dot u = \frac12 [ x , \xi ]$ and $u(0)=0$, then
	\begin{align*}
	u(t)
	&= \int_0^t\dot u(s) \dd s
	= \int_0^t \frac{[x(s),\xi(s)]}{2} \dd s \\
	&= \int_0^t \frac12 \left[ 
		\sum_{a=0}^\infty \frac{s^{a+1} J_{\mu_0}^a}{(a+1)!} \xi_0 , 
		\sum_{b=0}^\infty \frac{s^bJ_{\mu_0}^b}{b!} \xi_0 
		\right] \dd s \\
	&=  \sum_{a=0}^\infty\sum_{b=0}^\infty \frac12
		\left[ \frac{ J_{\mu_0}^a}{(a+1)!} \xi_0 , 
			\frac{J_{\mu_0}^b}{b!} \xi_0 \right] 
		\int_0^t s^{a+b+1}\dd s \\
	&= \sum_{a=0}^\infty\sum_{b=0}^\infty \frac{[J_{\mu_0}^a\xi_0,J_{\mu_0}^b\xi_0]}{2(a+1)!b!(a+b+2)} t^{a+b+2}\\
	&= t^2 \sum_{k=0}^\infty \left(\sum_{m=0}^k \frac{[J_{\mu_0}^m\xi_0,J_{\mu_0}^{k-m}\xi_0]}{2(m+1)!(k-m)!(k+2)} \right) t^{k}
	\end{align*}
	Using $B_k$ as defined in~\eqref{eq63ee2a7c}, notice that $B_0 = 0$, thus $u(t) = t^3 \sum_{k=1}^\infty B_k t^{k-1}$.
	
	We claim that the series defining $u$ is absolutely convergent for all $t\in\R$.
	Indeed, if $\|\cdot\|$ is any norm on $V$ and
	$C>0$ is so that $\|[x,y]\|\le C\|x\|\cdot\|y\|$ for all $x,y\in V_1$, 
	then
	\begin{align*}
	\|B_k\| 
	&\le  \frac{C \|J_{\mu_0}\|^k \|\xi_0\|^2}{2(k+2)} \sum_{m=0}^k \frac{1}{(m+1)!(k-m)!} \\
	&=  \frac{C \|J_{\mu_0}\|^k \|\xi_0\|^2}{2(k+2)(k+1)!} \sum_{m=0}^k \binom{k+1}{m+1} \\
	&=  \frac{C \|J_{\mu_0}\|^k \|\xi_0\|^2}{2(k+2)!} \left( \sum_{m=0}^{k+1} \binom{k+1}{m} - 1\right) \\
	&=  \frac{C \|J_{\mu_0}\|^k \|\xi_0\|^2}{2(k+2)!} \left( (1+1)^{k+1} - 1\right) \\
	&=  \frac{C \|J_{\mu_0}\|^k \|\xi_0\|^2(2^{k+1}-1)}{2(k+2)!} .
	\end{align*}
	Therefore, the series defining $u$ is absolutely convergent for all $t\in\R$.
\end{proof}

\begin{theorem}\label{thm63d04190}
	The analytic function $\sre:\frk g\to\frk g$ defined by
	\[
	\sre(\xi,\mu) 
	= \left( \sum_{k=0}^\infty \frac{J_\mu^k}{(k+1)!}\xi , \sum_{k=1}^\infty B_k(\xi,\mu) \right) 
	\]
	with
	\[
	 B_k(\xi,\mu) := \frac1{2(k+2)} \sum_{m=0}^k \frac{[J_\mu^m\xi,J_\mu^{k-m}\xi]}{(m+1)!(k-m)!} 
	\]
	is the sub-Riemannian exponential map, 
	after the identification $\frk g\simeq\frk g^* = T^*_{e}G$ via the scalar product 
	that extends the given one on $V_1$ and such that $V_1$ and $V_2$ are orthogonal,
	and after the identification $\frk g\simeq G$ via the group exponential map.
\end{theorem}
\begin{proof}
	This is a direct consequence of Proposition~\ref{prop63ee3157}.
\end{proof}

\subsection{The differential of the Exponential map}
\begin{theorem}\label{thm642a9e5d}
	The differential of the function $\sre$
	from Theorem~\ref{thm63d04190}
	at a point $(\xi,\mu)\in V_1\oplus V_2$ in the direction $(v,\nu)\in V_1\oplus V_2$, that is $D\sre(\xi,\mu)[v,\nu]$, is
	\begin{equation*}\label{eq63fef7fa}
	\begin{pmatrix}
	\sum_{k=0}^\infty\frac1{(k+1)!}\left(
		J_\mu^kv + 
		\sum_{m=1}^k J_\mu^{m-1} J_{\nu} J_\mu^{k-m} \xi \right) \\
	\sum_{k=1}^\infty \frac1{2(k+2)} \sum_{m=0}^k \frac{k-2m}{(m+1)!(k-m+1)!}
		\left[
			J_\mu^mv +
			\sum_{j=1}^m J_\mu^{j-1}J_{\nu} J_\mu^{m-j}\xi
		,J_\mu^{k-m}\xi
		\right] 
	\end{pmatrix}
	,
	\end{equation*}
	where we use the conventions $\sum_{j=1}^0 = 0$ and $J_0^0=\Id$.
\end{theorem}
\begin{proof}
	Using the notation of Proposition~\ref{prop63ee3157},
	we write $\sre(\xi,\mu) = (x(\xi,\mu),u(\xi,\mu))$.
	The derivative $\frac{\de x}{\de\xi}(\xi,\mu):V_1\to V_1$ is the linear map
	\[
		\frac{\de x}{\de\xi}(\xi,\mu) = \sum_{k=0}^\infty \frac{J_\mu^k}{(k+1)!} .
	\]
	An elementary computation shows that the derivative $\frac{\de x}{\de \mu}$ at $(\xi,\mu)$ is the linear map $V_2\to V_1$ 
	\[
	\frac{\de x}{\de \mu}(\xi,\mu) : \nu \mapsto \sum_{k=1}^\infty \frac1{(k+1)!} \left(\sum_{m=1}^k J_\mu^{m-1} J_\nu J_\mu^{k-m} \right)\xi  .
	\]	

	The derivatives of the second component of $\sre$ are
	\[
	\frac{\de u}{\de\xi}(\xi,\mu) = \sum_{k=1}^\infty \frac{\de }{\de\xi}B_k(\xi,\mu),
	\qquad
	\frac{\de u}{\de\mu}(\xi,\mu) = \sum_{k=1}^\infty \frac{\de }{\de\mu}B_k(\xi,\mu) .
	\]
	Since $B_k(\xi,\mu)$ is bilinear in $\xi$,
	the derivative $\frac{\de }{\de\xi}B_k(\xi,\mu)$ is the linear map $V_1\to V_2$ that maps $v\in V_1$ to
	\begin{align*}
	\frac{\de }{\de\xi}B_k(\xi,\mu)[v]
	&= \sum_{m=0}^k \frac{[J_\mu^mv,J_\mu^{k-m}\xi]+[J_\mu^m\xi,J_\mu^{k-m}v]}{2(m+1)!(k-m)!(k+2)} \\
	&= \frac1{2(k+2)} \sum_{m=0}^k \frac{k-2m}{(m+1)!(k-m+1)!}[J_\mu^mv,J_\mu^{k-m}\xi] .
	\end{align*}
	We can compute the derivative $\frac{\de }{\de\mu}B_k(\xi,\mu)$ as a linear map $V_2\to V_2$ that maps
	 $\nu\in V_2$ to
	\begin{multline*}
	\frac{\de }{\de\mu}B_k(\xi,\mu)[\nu]
	= \sum_{m=0}^k \frac1{2(m+1)!(k-m)!(k+2)} \times \\\times \bigg( 
	\left[(\sum_{j=1}^m J_\mu^{j-1}J_\nu J_\mu^{m-j})\xi,J_\mu^{k-m}\xi\right] + \left[J_\mu^m\xi,(\sum_{j=1}^{k-m} J_\mu^{j-1}J_\nu J_\mu^{k-m-j})\xi \right]
	\bigg) \\
	= \frac1{2(k+2)} \sum_{m=0}^k \frac{k-2m}{(m+1)!(k-m+1)!} \left[\sum_{j=1}^m J_\mu^{j-1}J_\nu J_\mu^{m-j}\xi,J_\mu^{k-m}\xi\right] .
	\end{multline*}
\end{proof}

\subsection{Jacobian of the sub-Riemannian exponential map}

\begin{definition}\label{def63f121f2}
	For every fixed pair $(\xi,\mu)\in V_1\oplus V_2$, 
	we define the following vector spaces.
	First, we define the following increasing sequence $U^\ell$ of subspaces of $V_1$:
	for $\ell=0$ we set 
	$U^0:=\{0\}\subset V_1$,
	for $\ell =1$ we set $U^1:=\R\xi$,
	and for $\ell > 1$,
	\[
		U^\ell := \Span \{\xi, J_\mu \xi, \dots, J_\mu^{\ell - 1} \xi \} \subset V_1.
	\]
	Second, we take the dual decreasing sequence $U_\ell$ of subspaces of $V_2$:
	for all $\ell\ge0$, set
	\[
	U_\ell := \{\nu \in V_2 : J_\nu (U^\ell) = \{0\}\} .
	\]
	Third, we define the orthogonal splitting 
	$V_2 = \bigoplus_{j=0}^\infty W_j \oplus W_\infty$
	by setting
	\[
	W_\infty:= \left\{\nu \in V_2 : 
			J_\nu(J_\mu^{\ell}\xi) = 0, \forall  \ell \ge 0
			\right\}.
	\]
	and by requiring
	\[
	U_\ell = U_{\ell+1} \oplus W_\ell .
	\]
\end{definition}

For example, $W_0$ is the orthogonal complement of $U_1 = \{\nu\in V_2: J_\nu\xi=0\}$,
while $W_1$ is the orthogonal complement of $U_2 = \{\nu\in V_2: J_\nu\xi = J_\nu J_\mu\xi=0\}$ in $U_1$.
See Section~\ref{se1} for an explicit example in the Heisenberg group.

Notice that
\begin{equation}\label{eq63ff0062}
	\text{if $\nu\in W_\ell\setminus\{0\}$, then $J_\nu(J_\mu^\ell\xi)\neq0$} .
\end{equation}

Since $V_2$ is finite dimensional, only finitely many $W_j$'s are non-trivial.
For $(\xi,\mu)$ fixed, define 
\[
d := \max\{\ell : W_\ell \ne \{0\}\} \in\N\cup\{\infty\},
\]
and
\begin{equation}\label{eq63fef49c}
	N_{\sre}(\xi,\mu) := 2 \sum_{j=0}^d j\dim(W_j) .
\end{equation}
If $W_\infty\neq\{0\}$, then $d=\infty$ and $N_{\sre}(\xi,\mu)=\infty$.
Notice also that $N_{\sre}(\xi,\mu)=N_{\sre}(\xi,s\mu)$ for all $s\neq0$.

\begin{theorem}\label{thm63f122b7}
	Let $G$ be a step-two Carnot group as above, and fix $(\xi,\mu)\in V$.
	
	\begin{enumerate}[label=(\alph*)]
		\item\label{item642aaa82_a}
		If $\nu \in W_\infty$, then $D\sre(\xi,\mu)[0,\nu]=0$.
		In particular, if $W_\infty \ne \{0\}$ then $\Jac(\sre)(\xi,\epsilon\mu)=0$ for all $\epsilon\in\R$.

	\item\label{item642aaa82_d}
		If $W_\infty = \{0\}$, then, for all $\epsilon\in\R$,
		\begin{equation}\label{eq63f122bf}
		\Jac(\sre)(\xi,\epsilon\mu) = \epsilon^{N_\sre(\xi,\mu)} \det(a(\epsilon)) ,
		\end{equation}
		where
		$a(\epsilon)$ is a $n\times n$ matrix, depending on $(\xi,\mu)$, analytic in $\epsilon$ and with $\det(a(0))>0$,
		and where $N_\sre(\xi,\mu)$ is as in~\eqref{eq63fef49c}.
	\end{enumerate}

	In conclusion,
	\begin{equation}\label{eq63ff0371}
		\Kmin(\xi,\mu) = N_\sre(\xi,\mu)
	\end{equation}
	where $\Kmin$ was defined in~\eqref{eq6400a110}.
\end{theorem}
A direct consequence of Theorem \ref{thm63f122b7} is the following corollary.

\begin{corollary} \label{step2}
In the setting of step-two Carnot groups, we have 
\[
 \Jac(\sre)(\xi) > 0, \qquad \forall \xi\in\scr D.
\]
\end{corollary}

\begin{remark}\label{rel2}
Fix a scalar product $\langle \cdot,\cdot \rangle$ on $\frk g$ as before and identify $\frk g$ with $\frk g^*$. Considering Remark \ref{rel1}, the value of $\Kmin(\xi,\mu)$ (or equivalently $\cal N_\lambda$) can also be computed out by geodesic growth vector $\cal G_\lambda$ with $\lambda = (\xi,\mu)$. See \cite[Definition 5.44]{MR3852258}, where ``ample'' is equivalent to strictly normal in our setting. In fact, it can be shown that the geodesic growth vector is given by $\cal G_\lambda = (k_1, \dots, k_{d + 2})$ with $k_{\ell + 1} = n - \dim (U_\ell)$ for $0 \le \ell \le d + 1$ and this approach gives the same number.
\end{remark}

\begin{proof}[Proof of Theorem~\ref{thm63f122b7}]
	By inspection of the formula in Theorem~\ref{thm642a9e5d}, one can easily check the statement~\ref{item642aaa82_a}. 

	Suppose now that $W_\infty = \{0\}$. 
	Define $A(\epsilon):=D\sre(\xi,\epsilon\mu)$, which is a linear map from $V$ to $V$.
	Using orthogonal projections in $V$, we decompose this linear maps into 
   $A^{V_1}_{V_1}(\epsilon):V_1\to V_1$, 
	$A^{V_1}_{W_\ell}(\epsilon):V_1\to W_\ell$, 
	$A^{W_\ell}_{V_1}(\epsilon):W_\ell\to V_1$, and
	$A^{W_r}_{W_s}(\epsilon):W_r\to W_s$, 
	for $0 \le \ell, r,s \le d$.
	For example, $A^{W_r}_{W_s}(\epsilon) = \pi_{s}\circ A(\epsilon)|_{W_r}$. Here $\pi_s$ denotes the orthogonal projection onto $W_s$.

	Notice that if $v\in V_1$, $\nu'\in W_\ell$, $k\ge1$, $0\le m\le k$ are such that
	\[
	0\neq\langle [J_\mu^m(v),J_\mu^{k-m}\xi],\nu' \rangle = -\langle J_\mu^m(v),J_{\nu'}J_\mu^{k-m}\xi \rangle ,
	\]
	then $k-m\ge \ell$, and thus $k\ge \ell$.
	Therefore, if $\ell\ge1$, we have
	\begin{align*}
	A^{V_1}_{W_\ell}(\epsilon)v 
	&= \pi_\ell\left(
		 \sum_{k=1}^\infty \frac{\epsilon^{k}}{2(k+2)} \sum_{m=0}^k \frac{k-2m}{(m+1)!(k-m+1)!}[J_\mu^m(v),J_\mu^{k-m}\xi]
		\right) \\
	&= \epsilon^\ell \pi_\ell\left(
		 \sum_{k=\ell}^\infty \frac{\epsilon^{k-\ell}}{2(k+2)} \sum_{m=0}^k \frac{k-2m}{(m+1)!(k-m+1)!}[J_\mu^m(v),J_\mu^{k-m}\xi]
		\right) ;
	\end{align*}
	while if $\ell=0$, then
	\begin{align*}
	A^{V_1}_{W_0}(\epsilon)v 
	&= \pi_0\left(
		 \sum_{k=1}^\infty \frac{\epsilon^{k}}{2(k+2)} \sum_{m=0}^k \frac{k-2m}{(m+1)!(k-m+1)!}[J_\mu^m(v),J_\mu^{k-m}\xi]
		\right) \\
	&= \epsilon \pi_0\left(
		 \sum_{k=1}^\infty \frac{\epsilon^{k-1}}{2(k+2)} \sum_{m=0}^k \frac{k-2m}{(m+1)!(k-m+1)!}[J_\mu^m(v),J_\mu^{k-m}\xi]
		\right) .
	\end{align*}
	
	Notice that if $\nu\in W_\ell$, $k\ge1$ and $1\le m\le k$ are such that $J_\nu J_\mu^{k-m}\xi \neq0$, then $k-m\ge\ell$, then $k\ge\ell+1$.
	Therefore
	\begin{align*}
	A^{W_\ell}_{V_1} (\epsilon)\nu 
	&= \sum_{k=1}^\infty \frac{\epsilon^{k-1}}{(k+1)!} \left(\sum_{m=1}^k J_\mu^{m-1} J_{(\nu)} J_\mu^{k-m} \right)\xi \\
	&= \epsilon^\ell \sum_{k=\ell+1}^\infty \frac{\epsilon^{k-\ell-1}}{(k+1)!} \left(\sum_{m=1}^k J_\mu^{m-1} J_{(\nu)} J_\mu^{k-m} \right)\xi .
	\end{align*}
	
	Notice that if $\nu\in W_r$, $\nu'\in W_s$, $k\ge1$, $1\le m\le k$ and $1\le j\le m$ are such that 
	\[
	0\neq\langle [J_\mu^{j-1}J_{(\nu)} J_\mu^{m-j}\xi,J_\mu^{k-m}\xi],\nu' \rangle
	= - \langle J_{\nu'}J_\mu^{k-m}\xi , J_\mu^{j-1}J_{(\nu)} J_\mu^{m-j}\xi \rangle ,
	\]
	then $J_\nu J_\mu^{m-j}\xi\neq0$ and $J_{\nu'}J_\mu^{k-m}\xi\neq0$, which implies $k-m\ge s$ and $m-j\ge r$, then $m\ge r+1$ and $k\ge r+s+1$.
	Therefore, 
	\begin{align*}
	A^{W_r}_{W_s}(\epsilon)\nu
	&= \pi_s \left( 
		\sum_{k=1}^\infty \frac{\epsilon^{k-1}}{2(k+2)} \sum_{m=1}^k \frac{k-2m}{(m+1)!(k-m+1)!} \left[(\sum_{j=1}^m J_\mu^{j-1}J_{\nu} J_\mu^{m-j})\xi,J_\mu^{k-m}\xi\right] 
		\right) \\
	&= \epsilon^{r+s} \pi_s\left( 
		\sum_{k=r+s+1}^\infty \frac{\epsilon^{k-r-s-1}}{2(k+2)} \sum_{m=1}^k \frac{k-2m}{(m+1)!(k-m+1)!} \left[(\sum_{j=1}^m J_\mu^{j-1}J_{\nu} J_\mu^{m-j})\xi,J_\mu^{k-m}\xi\right]
		\right) .
	\end{align*}
	
	Define the matrix $a(\epsilon)$ as
	\begin{align*}
		a^{V_1}_{V_1}(\epsilon) &= A^{V_1}_{V_1}(\epsilon) ,&
		a^{W_\ell}_{V_1}(\epsilon) &= A^{W_\ell}_{V_1}(\epsilon)/\epsilon^\ell ,\\
		a^{V_1}_{W_\ell}(\epsilon) &= A^{V_1}_{W_\ell}(\epsilon)/\epsilon^\ell ,&
		a^{W_r}_{W_s}(\epsilon) &= A^{W_r}_{W_s}(\epsilon)/\epsilon^{r+s} ,
	\end{align*}
	where $0\le\ell,r,s\le d$.
	Clearly, $\epsilon\mapsto a(\epsilon)$ is an analytic map with $a(0)$ given by
	\begin{equation}\label{eq63fef8a9}
	\begin{aligned}
	a^{V_1}_{V_1}(0) &= \Id_{V_1} \\
	a^{V_1}_{W_0}(0) &= 0 \\
	a^{V_1}_{W_\ell}(0) &= 
	\pi_\ell\left(
	\frac{1}{2(\ell+2)}  \frac{\ell}{(\ell+1)!}[(\cdot),J_\mu^{\ell}\xi] \right)
	= \frac{\ell}{2(\ell+2)!} \pi_\ell[(\cdot),J_\mu^{\ell}\xi]
	 \\
	a^{W_\ell}_{V_1}(0) &= \frac{1}{(\ell+2)!}  J_{(\cdot)} J_\mu^{\ell} \xi \\
	a^{W_r}_{W_s}(0) &= \frac{1}{2(r+s+3)} \frac{s-r-1}{(r+2)!(s+1)!} \pi_s\left[J_{(\cdot)} J_\mu^{r}\xi,J_\mu^{s}\xi\right] .
	\end{aligned}
	\end{equation}
	The proof that $\det(a(0))>0$ is long, we will show it in the next Section~\ref{sec642ab044}: 
	see Lemma~\ref{lem63ff0166}. 
	Finally the statement~\ref{item642aaa82_d} follows from the relation $\det(A(\epsilon)) = \epsilon^{N_\sre(\xi,\mu)} \det(a(\epsilon))$
\end{proof}

\subsection{The determinant of $a(0)$}\label{sec642ab044}

\begin{lemma}
	For $0 \le \ell,r,s\le d$,
define the following maps:
	\begin{align*}
		M_\ell:V_1\to W_\ell , &\quad M_\ell(v) := \pi_\ell[v,J_\mu^\ell\xi] , \\
		M^\ell:W_\ell\to V_1 , &\quad M^\ell(\nu) := -J_\nu J_\mu^\ell\xi , \\
		M^r_s:W_r\to W_s , &\quad M^r_s(\nu) := -\pi_s[J_\nu J_\mu^r\xi,J_\mu^s\xi] .
	\end{align*}
	The following identities hold:
	\begin{align*}
	(M_\ell)^* &= M^\ell , \\
	(M^r_s)^* &= M^s_r , \\
	M^r_s &=  M_s\circ M_r^* =  M_s\circ M^r ,
	\end{align*}
	where ${\cdot}^*$ denotes the conjugate with respect to the scalar product $\langle \cdot,\cdot\rangle$.
\end{lemma}
\begin{proof}
	Let $v\in V_1$ and $\nu\in W_\ell$.
	Then
	\[
	\langle M_\ell v,\nu \rangle
	= \langle [v,J_\mu^\ell\xi] , \nu \rangle
	= -\langle v , J_\nu J_\mu^\ell\xi \rangle
	= \langle v,M^\ell\nu \rangle .
	\]
	Therefore, $(M_\ell)^* = M^\ell$.
	
	Let $\nu_r\in W_r$ and $\nu_s\in W_s$.
	Then
	\[
	\langle M^r_s\nu_r,\nu_s \rangle
	= - \langle [J_{\nu_r} J_\mu^r\xi,J_\mu^s\xi] , \nu_s\rangle 
	= \langle J_{\nu_r} J_\mu^r\xi , J_{\nu_s}J_\mu^s\xi \rangle
	= - \langle \nu_r , [J_{\nu_s} J_\mu^s\xi,J_\mu^r\xi] \rangle
	= \langle \nu_r, M^s_r\nu_s \rangle .
	\]
	Therefore $(M^r_s)^* = M^s_r$.
	
	Let $\nu_r\in W_r$.
	Then
	\[
	M_s\circ M^r \nu_r
	= M_s(-J_{\nu_r} J_\mu^r\xi)
	= - \pi_s[ J_{\nu_r} J_\mu^r\xi ,J_\mu^s\xi]
	= M^r_s \nu_r .
	\]
	So, the last equality is also proved.
\end{proof}

\begin{lemma}\label{lem63ff0195}
	Define
	\[
		\scr M:=\left( \frac{M^r_s}{r+s+3} \right)_{r,s=0}^d : V_2\to V_2 .
	\]
	The matrix $\scr M$ is symmetric, positive definite and non-singular.
	In particular, $\det(\scr M)>0$.
\end{lemma}
\begin{proof}
	Notice that for all $\nu,\nu'\in V_2$,
	\[
	\langle \scr M\nu ,  \nu' \rangle
	= \sum_{r,s=0}^d \frac{ \langle M^r_s (\pi_r \nu), 
\pi_s \nu' \rangle }{r+s+3}   
	\]

	We apply Lemma~\ref{lem12190942} to the following spaces: $V_1$ with the scalar product $g=\langle \cdot,\cdot \rangle$ and $\R^{d+1}$ with the scalar product $h$ whose matrix with respect to the standard basis $(e_0,\dots,e_d)$ of $\R^{d+1}$ is 
	\[
	h_{ij}=h(e_i,e_j) = \frac1{i+j+3} = \frac1{(i+2)+(j+2)-1} ,
	\]
	which is a minor of a Hilbert matrix $\left(\frac1{i+j-1}\right)_{i,j\ge1}$.
	By Lemma~\ref{lem12190942}, the bilinear form $b=g\otimes h$ is a scalar product on $V_1\otimes\R^{d+1}$.
	
	Let $\bb M:V_2\to V_1\otimes\R^{d+1}$ be the map
	\[
	\bb M(\nu) = \sum_{\ell=0}^d M^\ell(\pi_\ell\nu)\otimes e_\ell .
	\]
	We claim that, for $\nu,\nu'\in V_2$,
	\begin{equation}\label{eq63fefe9a}
		\bb M^*(b)(\nu,\nu') = \langle \scr M\nu,\nu' \rangle .
	\end{equation}
	Indeed, 
	\begin{align*}
	\bb M^*(b)(\nu,\nu')
	&= b(\bb M\nu,\bb M\nu') \\
	&= \sum_{r,s=0}^d b(M^r(\pi_r\nu)\otimes e_r,M^s(\pi_s\nu')\otimes e_s) \\
	&= \sum_{r,s=0}^d \frac{\langle M^r(\pi_r\nu),M^s(\pi_s\nu') \rangle}{r+s+3} \\
	&= \sum_{r,s=0}^d \frac{\langle M_sM^r(\pi_r\nu),\pi_s\nu' \rangle}{r+s+3} \\
	&= \langle \scr M\nu,\nu' \rangle .
	\end{align*}
	
	Next, we claim that the map $\bb M$ is injective.
	Indeed, if $\nu\in V_2$ is such that $\bb M\nu=0$,
	then $0 = M^\ell(\pi_\ell\nu) = -J_{\pi_\ell\nu} J_\mu^\ell\xi$ for all $\ell$.
	From~\eqref{eq63ff0062}, we obtain that $\pi_\ell\nu=0$ for all $\ell$, and thus $\nu=0$.
	
	Since $b$ is positive definite and $\bb M$ is injective, 
	we obtain from~\eqref{eq63fefe9a} that $\scr M$ is symmetric, non-singular and positive definite.
	In particular, $\det(\scr M)>0$.
\end{proof}

The next lemma is a standard result from Linear Algebra.

\begin{lemma}\label{lem12181152}
	Let $A,B,C,D$ be matrices of suitable dimensions, or linear maps between suitable spaces.
	Suppose $A$ is invertible.
	Then
	\[
	\det \begin{pmatrix}A&B\\ C&D\end{pmatrix} 
	= \det(A)\cdot \det(D-CA^{-1}B) .
	\]
\end{lemma}

\begin{lemma}\label{lem63ff0166}
	The determinant of $a(0)$ is
	\[
	\det(a(0)) = \left(\prod_{\ell=0}^d \left(\frac{\ell+1}{(\ell+2)!}\right)^{\dim W_\ell} \right)^2 \cdot \det\left( \scr M \right) > 0 .
	\]
\end{lemma}
\begin{proof}
	We can rewrite the map $a(0)$ from~\eqref{eq63fef8a9} as follows:
	\begin{align*}
		a^{V_1}_{V_1}(0) &= \Id_{V_1} \\
		a^{V_1}_{W_\ell}(0) &= \frac{\ell}{2(\ell+2)!} M_\ell \\
		a^{W_\ell}_{V_1}(0) &= - \frac{1}{(\ell+2)!}  M^\ell \\
		a^{W_r}_{W_s}(0) &= \frac{1+r-s}{2(r+s+3)(r+2)!(s+1)!} M^r_s .
	\end{align*}

	Since
	\[
	a^{V_1}_{W_s}(0) a^{W_r}_{V_1}(0)
	= - \frac{s}{2(s+2)!} \frac{1}{(r+2)!} M_s\circ M^r
	= - \frac{s}{2(s+2)!(r+2)!}  M^r_s ,
	\]
	then we get from Lemma~\ref{lem12181152}
	\begin{align*}
	\det(a(0)) 
	&= \det\left(
		(a^{W_r}_{W_s}(0))_{r,s=0}^d - (a^{V_1}_{W_\ell}(0))_{\ell=0}^d (a^{W_\ell}_{V_1}(0))_{\ell=0}^d
		\right) \\
	&= \det\left(
		\left(\frac{1+r-s}{2(r+s+3)(r+2)!(s+1)!} M^r_s \right)_{r,s=0}^d 
		+ \left( \frac{s}{2(s+2)!(r+2)!}  M^r_s \right)_{r,s=0}^d
		\right) \\
	&= \det\left(
		\left( \frac{(r+1)(s+1)}{(r+s+3)(r+2)!(s+2)!} M^r_s \right)_{r,s=0}^d 
		\right) \\
	&= \left(\prod_{\ell=0}^d \left(\frac{\ell+1}{(\ell+2)!}\right)^{\dim W_\ell} \right)^2 \cdot \det\left( \left( \frac{M^r_s}{r+s+3} \right)_{r,s=0}^d \right) .
	\end{align*}
\end{proof}

\begin{lemma}[A side-note on tensor product of matrices]\label{lem12190942}
	Let $V$ and $W$ be a two vector spaces and let $g$ and $h$ be two bilinear maps on $V$ and $W$, respectively.
	Let  $b=g\otimes h$ be the bilinear map on $V\otimes W$ defined by $b(v_1\otimes w_1,v_2\otimes w_2) = g(v_1,v_2)h(w_1,w_2)$ for all $v_i\in V$ and $w_i\in W$.
	
	If $g$ and $h$ are symmetric and positive definite (i.e., scalar products), then $b$ is also symmetric and positive definite (i.e., a scalar product).
\end{lemma}
\begin{proof}
	Clearly $b$ is symmetric.
	
	Let $(e_1,\dots,e_d)$ be an orthonormal basis for $(W,h)$.
	Any element of $V\otimes W$ can be written as $\sum_{i=1}^d v_i\otimes e_i$ for some $v_i\in V$.
	Indeed $\sum_k u_k\otimes w_k = \sum_k u_k\otimes (\sum_{i=1}^d w_k^i e_i) = \sum_{i=1}^d (\sum_k w_k^i u_k)\otimes e_i$.
	So if $x=\sum_{i=1}^d v_i\otimes e_i\in V\otimes W$, then
	\[
	b(x,x)
	= b(\sum_{i=1}^d v_i\otimes e_i,\sum_{j=1}^d v_j\otimes e_j)
	= \sum_{i,j=1}^d g(v_i,v_j) h(e_i,e_j)
	= \sum_{i=1}^d g(v_i,v_i)
	\ge 0 ,
	\]
	and clearly $b(x,x)=0$ if and only if $x=0$.
	Therefore, $b$ is positive definite.
\end{proof}


\section{Examples of Carnot groups of step two}
\label{sec64c3b071}


In this section we collect several examples of step-two Carnot groups. 
In Section~\ref{se1} we recall the classical example of Heisenberg group. 
Then we give several generalizations of the Heisenberg group in Sections~\ref{se2}-\ref{se4}. 
In particular, in Section~\ref{se2} we give the examples of free step-two groups on which $N_{GEO} > 2Q - n$. 
Then in Section~\ref{se3} we give the main examples of the work: groups on which $N_{CE} > N_{GEO}$.
Finally in Section~\ref{se4} we provide more examples of step-two groups where the $N_{CE}$ can be computed.

\subsection{The Heisenberg group $\HH$}\label{se1}
Recall that the simplest non-abelian Carnot group is the Heisenberg group $\HH$ whose Lie algebra is given by $\frk g = V_1 \oplus V_2$ with 
\[
V_1 := \Span \{X_1, X_2\}, \qquad V_2 := \Span \{Y\}.
\]
Here $\{X_1, X_2\}$ is an orthonormal basis of $V_1$ and the only nontrivial bracket relation of $\frk g$ is $[X_1, X_2] = Y$. 
The topological dimension of $\HH$ is $n=3$ and the homogeneous dimension is $Q=4$.

By formula~\eqref{eq64c3d1c5}, we obtain that, for $a,b,c,d \in \R$ and $\mu\in V_2$,
\[
(ad - bc) \langle \mu, Y \rangle  = \langle \mu, [a X_1 + b X_2, c X_1 + d X_2] \rangle = \langle J_\mu (a X_1 + b X_2), c X_1 + d X_2 \rangle.
\]
Since $a,b,c,d \in \R$ are arbitrary, the matrix representation of $J_\mu$ with respect to the orthonormal basis $\{X_1, X_2\}$
 is
\begin{equation}\label{eq64c3c777}
J_\mu = \langle \mu, Y \rangle \begin{pmatrix}
0 & -1 \\
1 & 0
\end{pmatrix}.
\end{equation}

As a result, see Table~\ref{tab64db26c4}, it is easy to check with Theorem~\ref{thm63f122b7} that 
\begin{align*}
\Kmin(\xi,\mu) = \begin{cases}
      0 & \text{if $\xi \ne 0$},  \\
		\infty &\text{if $\xi = 0$}.  
		\end{cases}
\end{align*}
By Theorem~\ref{thm6401aa05}, the geodesic dimension of $\HH$ is 5, in accordance with the literature. Furthermore, it follows from the result of \cite{J09} that $N_{CE} = 5$ on $\HH$. See also Proposition \ref{CRCE} in Section \ref{se4}.

\newcolumntype{m}{>{$}c<{$}}
\begin{table}
\label{tab64db26c4}
{\footnotesize
\begin{tabular}{|m||m|m|m|}
\hline
 & \xi = 0 & \xi\neq0,\ \mu=0 & \xi\neq0,\ \mu\neq0 \\
\hline\hline
	 & U^0 = \{0\} & U^0 = \{0\} & U^0 = \{0\} \\
U^\ell	& U^1 = \{0\} & U^1 = \R\xi & U^1 = \R\xi \\
	& U^\ell = \{0\},\ \ell>1 & U^\ell = \R\xi,\ \ell>1 & U^\ell = V_1,\ \ell>1 \\
\hline
 & U_0=V_2 & U_0=V_2 & U_0=V_2 \\
U_\ell	& U_1 = V_2 & U_1 = \{0\} & U_1 = \{0\} \\
	& U_\ell = V_2,\ \ell>1 & U_\ell = \{0\},\ \ell>1 & U_\ell = \{0\},\ \ell>1 \\
\hline
W_\infty & W_\infty = V_2 & W_\infty = \{0\} & W_\infty = \{0\} \\
\hline
W_j,\ j<\infty & W_0 = \{0\} & W_0 = V_2 & W_0 = V_2 \\
	& W_j = \{0\},\ 1\le j<\infty & W_j = \{0\},\ 1\le j<\infty & W_j = \{0\},\ 1\le j<\infty \\
\hline
\Gamma(\xi,\mu) = N_{\sre} & \infty & 0 & 0 \\
\hline
\end{tabular}
}
\medskip 
\caption{We summarize the objects from Definition~\ref{def63f121f2} for the first Heisenberg group $\HH$.}
\end{table}

\subsection{Free step-two Carnot group with $k$ generators $N_{k,2}$}\label{se2}

One possible generalization of the Heisenberg group $\HH$ is the free step-two group $N_{k,2}$.
For every $k \ge 2$, the Lie algebra of $N_{k,2}$ is $\frk g = V_1 \oplus V_2$ with 
\[
V_1 := \Span \{X_1, \dots, X_k\}, \qquad V_2 := \Span \{Y_{1,2}, Y_{1,3}, \dots, Y_{k-1,k}\}.
\]
Here $\{X_1, \dots, X_k\}$ is an orthonormal basis of $V_1$ with the property $[X_i, X_j] = Y_{i,j}, \forall 1 \le i < j \le k$. For $k = 2$, $N_{2,2}$ is exactly the Heisenberg group. As before,
the formula \eqref{eq64c3d1c5} yields for every $\mu \in V_2$:
\[
\sum_{i < j} (a_i b_j - a_j b_i) \langle \mu, Y_{i,j} \rangle =  \langle \mu, [\sum_{i = 1}^k a_i X_i, \sum_{j = 1}^k b_j X_j] \rangle = \sum_{1 \le i,j \le k} a_i b_j \langle J_\mu X_i,  X_j \rangle
\]
and thus under the orthonormal basis $\{X_1, \dots, X_k\}$
\[
J_\mu = \begin{pmatrix}
0 & -\langle \mu, Y_{1,2} \rangle & -\langle \mu, Y_{1,3} \rangle& \cdots &  -\langle \mu, Y_{1,k} \rangle\\
\langle \mu, Y_{1,2} \rangle & 0  & -\langle \mu, Y_{2,3} \rangle& \cdots &  -\langle \mu, Y_{2,k} \rangle\\
\langle \mu, Y_{1,3} \rangle & \langle \mu, Y_{2,3} \rangle & 0 & \cdots & -\langle \mu, Y_{3,k} \rangle\\
\vdots & \vdots &  \vdots & \ddots & \vdots \\
\langle \mu, Y_{1,k} \rangle & \langle \mu, Y_{2,k} \rangle & \langle \mu, Y_{3,k} \rangle & \cdots & 0
\end{pmatrix}.
\]
Since $\{Y_{1,2}, Y_{1,3}, \dots, Y_{k-1,k}\}$ is a basis, the map $\mu \mapsto J_\mu$ gives a linear isomorphism between $V_2$ and $\so(V_1,\langle \cdot,\cdot \rangle)$. Fix a pair $(\xi,\mu)\in V_1\oplus V_2$. 
Recall that the spaces $U^\ell$, $U_\ell$ and $W_\ell$ are defined in Definition~\ref{def63f121f2}.

\begin{lemma}\label{chafreeWi}
For $(\xi,\mu)\in V_1\oplus V_2$ fixed, 
we have
\[
U_\ell \simeq \so((U^\ell)^\bot,\langle \cdot,\cdot \rangle), \quad
W_\ell \simeq  \so((U^\ell)^\bot,\langle \cdot,\cdot \rangle) / \so((U^{\ell + 1})^\bot,\langle \cdot,\cdot \rangle), \qquad \forall \ell \ge 0.
\]
On the right-hand side, we regard an element of $\so((U^{\ell + 1})^\bot,\langle \cdot,\cdot \rangle)$ as an element of $\so((U^\ell)^\bot,\langle \cdot,\cdot \rangle)$ by zero extension.
\end{lemma}

\begin{proof}
As in Definition~\ref{def63f121f2}, $U_\ell := \{\nu \in V_2 : J_\nu (U^\ell) = \{0\}\}, \forall \ell \ge 0$. 
It is easy to see that if $\nu \in U_\ell$, then $J_\nu$ maps $(U^\ell)^\bot$ to itself, or equivalently $J_\nu|_{(U^\ell)^\bot} \in \so((U^\ell)^\bot,\langle \cdot,\cdot \rangle)$.
Conversely, if we start from an element in $ \so((U^\ell)^\bot,\langle \cdot,\cdot \rangle)$, by zero extension we obtain a unique element $\nu \in U_\ell$. This gives us the isomorphism between $U_{\ell}$ and $\so((U^{\ell})^\bot,\langle \cdot,\cdot \rangle)$. Then the rest of the lemma follows from the following commutative diagram:
\[\begin{tikzcd}
	{U_{\ell + 1}} & {\so((U^{\ell + 1})^\bot,\langle \cdot,\cdot \rangle)} \\
	{U_{\ell}} & {\so((U^\ell)^\bot,\langle \cdot,\cdot \rangle)}
	\arrow["\simeq", from=1-1, to=1-2]
	\arrow[from=1-1, to=2-1]
	\arrow["\simeq", from=2-1, to=2-2]
	\arrow[from=1-2, to=2-2]
\end{tikzcd}\]
where the map from $\so((U^{\ell + 1})^\bot,\langle \cdot,\cdot \rangle)$ to $\so((U^\ell)^\bot,\langle \cdot,\cdot \rangle)$ is given by zero extension.
\end{proof}

\begin{lemma}\label{chafreeinWi}
Fix a pair $(\xi,\mu)\in V_1\oplus V_2$. If $U^\ell = U^{\ell + 1}$, then $U^j = U^\ell$ for every $j \ge \ell$.
\end{lemma}

\begin{proof}
If $U^\ell = U^{\ell + 1}$, then $J_\mu^\ell \xi = \sum_{i = 0}^{\ell - 1} a_i J_\mu^i \xi$ for some $a_i \in \R$. Therefore, $J_\mu^{\ell + 1} \xi = \sum_{i = 0}^{\ell - 1} a_i J_\mu^{i + 1} \xi \in U^{\ell + 1} = U^\ell$, i.e., $U^{\ell + 2} = U^\ell$. By induction, we obtain the lemma.
\end{proof}

\begin{proposition}
On free step-two group with $k$ generators $N_{k,2}$ with $k \ge 3$, we have 
\[2Q - n = \frac{3k^2 - k}{2} < \frac{3k^2 - k}{2} + \frac{k(k - 1)(k - 2)}{3} = N_{GEO} \le N_{CE}. \]
\end{proposition}

\begin{proof}
For every $\ell\ge0$, if $U^\ell = U^{\ell + 1}$, then 
$W_\infty = U_\ell \simeq \so((U^\ell)^\bot,\langle \cdot,\cdot \rangle)$,
by Lemmas~\ref{chafreeWi} and~\ref{chafreeinWi}.
Since $\dim (U^{\ell}) \le \ell$, 
if there is $\ell < k - 1$ with $U^\ell = U^{\ell + 1}$, then 
$W_\infty\neq\{0\}$. 

Hence, if  $W_\infty = \{0\}$, then
$\dim (U^{\ell}) = \ell$ for all $\ell < k$,
and $\dim(U^\ell)\in\{k,k-1\}$ for all $\ell \ge k$.
It follows that $\dim(W_{\ell})=0$ for $\ell\ge k-1$
and, for $0 \le \ell < k - 1$,

\begin{align*}
&\dim(W_\ell) = \dim(\so((U^\ell)^\bot,\langle \cdot,\cdot \rangle)) - \dim(\so((U^{\ell + 1})^\bot,\langle \cdot,\cdot \rangle)) \\
=& \frac{(k - \ell)(k - \ell - 1)}{2} -  \frac{(k - \ell - 1)(k - \ell - 2)}{2} = k - \ell -1, 
\end{align*}
and thus $d = k - 2$.
As a result, by \eqref{eq63ff0371} and Theorem \ref{thm6401aa05},
\begin{align*}
N_{GEO} = \frac{3k^2 - k}{2} + \sum_{\ell = 1}^{k - 2} 2\ell (k - \ell -1)=  \frac{3k^2 - k}{2} + \frac{k(k - 1)(k - 2)}{3}.
\end{align*}
\end{proof}

\subsection{Step-two groups induced by star graphs $K_{1,k}$}\label{se3}
In \cite{DCDMM18} the following step-two stratified Lie algebras are associated to star-shaped graphs.
For every $k \ge 1$, the Lie algebra is given by $\frk g = V_1 \oplus V_2$ with 
\[
V_1 := \Span \{X_0, X_1 \dots, X_k\}, \qquad V_2 := \Span \{Y_1,  \dots, Y_k\} ,
\]
where the nontrivial bracket relations are $[X_0, X_j] = Y_j$, for all $1 \le j \le k$. 
We fix a scalar product on $\frk g$ such that 
 $\{X_0, X_1, \dots, X_k,Y_1,  \dots, Y_k\}$ is an orthonormal basis.
We remark that $k = 1$ case corresponds to the Heisenberg group $\HH$. 

By formula \eqref{eq64c3d1c5}, if $\sum_{i=0}^k a_i X_i , \sum_{j=0}^k b_j X_j\in V_1$ and $\mu\in V_2$, then 
\[
\sum_{j=1}^k (a_0b_j-a_jb_0) \langle \mu,Y_j \rangle
= \langle \mu , [\sum_{i=0}^k a_i X_i , \sum_{j=0}^k b_j X_j ]  \rangle
= \sum_{0 \le i,j \le k} a_ib_j \langle J_\mu X_i, X_j \rangle .
\]
Therefore, for every $\mu\in V_2$ we have
\[
J_\mu = \begin{pmatrix}
0 & -\langle \mu, Y_{1} \rangle & \cdots &  -\langle \mu, Y_{k} \rangle\\
\langle \mu, Y_{1} \rangle & & & \\
\vdots & & 0 & \\
\langle \mu, Y_{k} \rangle & & &
\end{pmatrix}
\]
with respect to the orthonormal basis $\{X_0,X_1, \dots, X_k\}$.
The following proposition answers a question posed by Rizzi in~\cite{MR3502622},
that is, it shows that there are sub-Riemannian Carnot groups such that $N_{GEO} \neq N_{CE}$.
 
\begin{proposition}
In the framework of step-two groups induced by star graphs $K_{1,k}$ with $k \ge 2$, the following is true:
\begin{equation}\label{eq64970cf5}
\Kmin(K_{1,k}) = 0 < 2k - 2 = \Kmax(K_{1,k}).
\end{equation}
In particular, we have $N_{GEO} < N_{CE}$ in $K_{1,k}$ with $k \ge 2$
\end{proposition}
\begin{proof}
	Using the above bases for $V_1$ and $V_2$, we write $\xi = (\xi_0,\hat\xi)$ with $\hat\xi \in\R^k$,
	and $\mu\in\R^k$.
	Then
	\[
	J_\mu\xi = (-\mu\cdot\hat\xi , \xi_0\mu) \in V_1,
	\]
	where $\cdot$ denotes the standard scalar product on $\R^k$. To compute $\Kmin(\xi,\mu)$, we consider three cases.
	First, if $\xi_0\neq0$, then $U_1=\{\nu \in V_2 :  J_\nu\xi = 0\} = \{0\}$.
	This implies $W_\infty = \{0\}$, $W_0=V_2$ and $d = 0$.
	
	Second, consider the case $\xi_0 = 0$ and $\mu \cdot \hat\xi \neq 0$. 
	For a similar reason we have
	that $U_1$ has codimension 1 in $V_2$ and
	$U_2 = \{\nu \in V_2 : J_\nu J_\mu \xi=  J_\nu\xi = 0\} = \{0\}$.
	This implies $W_\infty = \{0\}$ as  well, but in this case $d = 1$, 
	$\dim(W_0)=1$, and $\dim(W_1) = k -1$.
	
	In the remaining third case,
	when $\xi_0 = 0$ and $\mu \cdot \hat\xi = 0$,
	we have $J_\mu \xi = 0$ and thus $W_\infty \ne \{0\}$.
	 
	In conclusion, from Theorem~\ref{thm63f122b7} we obtain
	\begin{align*}
	\Kmin(\xi,\mu) = \begin{cases}
	      0 & \text{if $\xi_0 \ne 0$},  \\
	      2k -2 &\text{if $\xi_0 = 0$ and $\mu \cdot \hat\xi \ne 0$}, \\
			\infty &\text{if $\xi_0 = 0$ and $\mu \cdot \hat\xi = 0$}.
			\end{cases}
	\end{align*}
	This implies~\eqref{eq64970cf5}, and the we conclude by
	Theorems \ref{thm6401aa05} and \ref{thm6419896b}.
\end{proof}

\subsection{Step-two groups $G_A$}\label{se4}

In this section we introduce a subclass of step-two groups which are again generalizations of Heisenberg group $\HH$ but not ideal Carnot groups except for very special cases.  Given a matrix of full-rank $A = A_{m \times k} = (A_{ij})_{1 \le i \le m, 1 \le j \le k}$ with $m \le k$, the Lie algebra of $G_A$ is given by $\frk g = V_1 \oplus V_2$ with 
\[
V_1 := \Span \{X_1, X_2, \dots, X_{2k - 1}, X_{2k}\}, \qquad V_2 := \Span \{Y_1,  \dots, Y_m\},
\]
where the nontrivial relations of the Lie algebra $\frk g$ are $[X_{2j - 1}, X_{2j}] = \sum_{i = 1}^m A_{ij} Y_{i}, \forall 1 \le j \le k$. The topological dimension of $G_A$ is $n= 2k + m$ and the homogeneous dimension is $Q= 2k + 2m$. This subclass of step-two groups is associated to CR manifolds, see \cite{NRS01}. We remark that for the case $m = k = 1$ and $A = 1$, it is nothing but the Heisenberg group $\HH$.

We fix a scalar product on $\frk g$ such that  $\{X_1, X_2, \dots, X_{2k - 1}, X_{2k}, Y_1,  \dots, Y_m\}$ is an orthonormal basis. In the following, using the basis above we write $\xi = (\xi_1, \dots, \xi_{2k}) \in \R^{2k}$ and $\mu = (\mu_1, \dots, \mu_m) \in \R^m$. Moreover, we use to denote $\cdot$ the standard scalar product on $\R^m$ and $A_j := (A_{1j}, \dots, A_{mj})$ the $j$-th column of the matrix $A$. Using these notations, and by formula~\eqref{eq64c3d1c5} again, we have
\[
J_\mu = \diag \left\{ \begin{pmatrix} 0 & - \mu \cdot A_1  \\
\mu \cdot A_1  & 0 \end{pmatrix}, \ldots, \begin{pmatrix} 0 & - \mu \cdot A_k \\
\mu \cdot A_k  & 0 \end{pmatrix} \right\}.
\]

Observe that from the matrix above, we have $J_\nu J_\mu = J_\mu J_\nu$ and as a consequence $W_\infty = \{\nu \in V_2 : J_\nu \xi = 0\}$. Thus we obtain
\begin{align*}
\Kmin(\xi,\mu) = \begin{cases}
      0 & \text{if $\{\nu \in V_2 : J_\nu \xi = 0\} = \{0\}$},  \\
		\infty &\text{if $\{\nu \in V_2 : J_\nu \xi = 0\} \ne \{0\}$},  
		\end{cases}
\end{align*}
which implies $N_{GEO} = 2Q - n = 2k + 3m$ by Theorem~\ref{thm6401aa05}.

Now we compute the explicit formula of the Jacobian of the sub-Riemannian exponential map with respect to the basis $\{X_1, \dots, X_{2k}, Y_1, \dots , Y_m\}$, and we use this formula to compute the curvature exponent. For the computation on Heisenberg groups, we refer to \cite{J09, BKS18}.
In fact, using the matrix of $J_\mu$ above, and deducing from Lemma \ref{lem63d454ee} again, the sub-Riemannian exponential map is represented by
\[
\sre(\xi, \mu) = \begin{pmatrix}
\E_1(\xi_1,\xi_2, \mu \cdot A_1) \\
 \vdots \\
\E_1(\xi_{2k - 1},\xi_{2k}, \mu \cdot A_k) \\
A \begin{pmatrix}
\E_2(\xi_1,\xi_2, \mu \cdot A_1) \\
 \vdots \\
\E_2(\xi_{2k - 1},\xi_{2k}, \mu \cdot A_k) 
\end{pmatrix}
\end{pmatrix},
\]
where
\[
\E_1(v_1, v_2, \nu) := \begin{pmatrix}
\frac{\sin{\nu}}{\nu} v_1 - \frac{1 - \cos{\nu}}{\nu} v_2 \\
\frac{1 - \cos{\nu}}{\nu} v_1 + \frac{\sin{\nu}}{\nu} v_2 
\end{pmatrix}
\quad \mbox{and} \quad 
\E_2(v_1, v_2, \nu) := \frac{1}{2} \frac{\nu - \sin{\nu}}{\nu^2} (v_1^2 + v_2^2).
\]

Define
\begin{align*}
\wt \J_{1,1} (\mu) := \diag \{   \J_{1,1}(\mu \cdot A_1), \dots,  \J_{1,1}(\mu \cdot A_k) \},& \\
\wt  \J_{i,j} (\xi, \mu) := \diag \{  \J_{i,j}(\xi_1, \xi_2, \mu \cdot A_1), \dots,   \J_{i,j}(\xi_{2k - 1}, \xi_{2k}, \mu \cdot A_k) \},& \ \forall  (i,j) \ne (1,1)
\end{align*}
with
\begin{align*} 
\J_{1,1}(\nu) &:= \begin{pmatrix} 
\frac{\sin{\nu}}{\nu} & - \frac{1 - \cos{\nu}}{\nu} \\
\frac{1 - \cos{\nu}}{\nu} & \frac{\sin{\nu}}{\nu}
\end{pmatrix}, \\
\J_{1,2}(v_1, v_2, \nu) &:= \begin{pmatrix}  
\frac{\nu \cos{\nu} - \sin{\nu}}{\nu^2} v_1 - \frac{\nu \sin{\nu} - 1 + \cos{\nu}}{\nu^2} v_2 \\
\frac{\nu \sin{\nu} - 1 + \cos{\nu}}{\nu^2} v_1 + \frac{\nu \cos{\nu} - \sin{\nu}}{\nu^2} v_2
\end{pmatrix}, \\
\J_{2,1}(v_1, v_2, \nu) &:=  \begin{pmatrix}  \frac{\nu - \sin{\nu}}{\nu^2} v_1 &  \frac{\nu - \sin{\nu}}{\nu^2} v_2 \end{pmatrix}, \\
\J_{2,2}(v_1, v_2, \nu) &:= \frac{1}{2} \frac{2\sin{\nu} - \nu - \nu \cos{\nu}}{\nu^3} (v_1^2 + v_2^2).
\end{align*}

Then the differential is presented by
\[
D\sre (\xi, \mu) = \begin{pmatrix}
\wt \J_{1,1} (\mu) & \wt \J_{1,2} (\xi, \mu) A^T \\
A \wt \J_{2,1} (\xi, \mu) & A \wt \J_{2,2} (\xi, \mu) A^T
\end{pmatrix},
\]
where $T$ denotes the transpose of the matrix.
In the computation below we assume $\wt \J_{1,1} (\mu)$ is invertible and the final formula \eqref{JacCr} holds for all $(\xi,\mu)$ by continuity. 
Then Lemma \ref{lem12181152} gives
\[
\Jac(\sre)(\xi, \mu) =  \det (A \wt \J(\xi, \mu) A^{T}) \prod_{j = 1}^k \det(\J_{1,1}(\mu \cdot A_j)), 
\]
where
\[
\wt \J(\xi, \mu) := \diag \{  \J(\xi_1, \xi_2, \mu \cdot A_1), \dots, \J(\xi_{2k - 1}, \xi_{2k}, \mu \cdot A_k)  \} 
\]
with
\begin{align*}
\J(v_1, v_2, \nu) :=& \J_{2,2}(v_1, v_2, \nu) - \J_{2,1}(v_1, v_2, \nu)\J_{1,1}(\nu)^{-1} \J_{1,2}(v_1, v_2, \nu). 
\end{align*}

In fact, we can write down the explicit formulas for $\J(\cdot)$ and $\det(\J_{1,1}(\cdot))$:
\begin{align*} 
\J(v_1, v_2, \nu)  =    \frac{f_1(\frac{\nu}{2})}{4 f_2(\frac{\nu}{2})}(v_1^2 + v_2^2),  \quad \mbox{and} \quad
 \det(\J_{1,1}(\nu)) = f_2\left(\frac{\nu}{2}\right)^2
\end{align*}
with two auxiliary functions $f_1$ and $f_2$ defined by
\begin{align}\label{deff12}
f_1(s) := \frac{\sin{s} - s \cos{s}}{s^3}, \qquad f_2(s) := \frac{\sin{s}}{s}.
\end{align}

Now we need the following lemma from Linear Algebra.

\begin{lemma}[Cauchy--Binet formula]\label{cbfor}
	Let $A = (A_1, \dots, A_k)$ be an $m \times k$ matrix and $B = \begin{pmatrix} B_1 \\ \vdots \\ B_k \end{pmatrix}$ a $k \times m$ matrix with $m \le k$. Then  
	\[
	\det (AB) = \sum_{1 \le i_1 < \dots < i_m \le k} \det(A_{i_1}, \dots, A_{i_m}) \det(B_{i_1}^T, \dots, B_{i_m}^T).
	\]
\end{lemma}

Applying Lemma~\ref{cbfor} to $B = \wt \J(\xi, \mu) A^{T}$, we obtain
\begin{align} \label{JacCr}
\Jac(\sre)(\xi, \mu) = 4^{-m} \sum_{1 \le i_1 < \dots < i_m \le k} \det(A_{i_1}, \dots, A_{i_m})^2 \J^{i_1, \dots, i_m}(\xi, \mu),
\end{align}
where $\J^{i_1, \dots, i_m}(\xi, \mu)$ is defined by
\[
 \prod_{j \notin \{i_1, \dots, i_m\} } f_2\left(\frac{\mu \cdot A_j}{2}\right)^2  \prod_{j \in \{i_1, \dots, i_m\} } f_1\left(\frac{\mu \cdot A_j}{2}\right)f_2\left(\frac{\mu \cdot A_j}{2}\right) (\xi_{2j - 1}^2 + \xi_{2j}^2)   .
\]

\begin{lemma}\label{chaD}
On step-two groups $G_A$, under the basis $\{X_1, \dots, X_{2k}, Y_1, \dots , Y_m\}$, the set $\scr D$ in Definition \ref{def6401a951} satisfies
\[
\scr D \subset \R^{2k} \times \{\mu : |\mu \cdot A_j | < 2 \pi, \forall 1 \le j \le k\}.
\]
\end{lemma}

\begin{proof}
From definition it suffices to prove that if $|\mu \cdot A_j | = 2 \pi$ for some $j \in \{1, \dots, k\}$, then $\Jac(\sre)(\xi, \mu) = 0$. Without loss of generality, we assume that $|\mu \cdot A_1 | = 2 \pi$. In fact, it follows from \eqref{deff12} that $f_2(\pm \pi) = 0$, which implies $\Jac(\sre)(\xi, \mu) = 0$ by \eqref{JacCr}.
\end{proof}

\begin{proposition}\label{CRCE}
On step-two groups $G_A$, $N_{CE} = 2Q-n = 2k + 3m$.
\end{proposition}

\begin{proof}
By \ref{item64267a2a_2} of Proposition \ref{prop6401afb9} as well as Corollary~\ref{step2}, we only need to prove 
\[
\Jac(\sre)(\xi, \lambda \mu) \ge \Jac(\sre)(\xi, \mu), \quad \forall (\xi, \mu) \in \scr D, \lambda \in [0,1].
\]
In fact, noticing that the even function $f_2$ is decreasing on $[0,\pi]$, we have 
$f_2(\lambda s) \ge f_2(s)$ for $s \in [-\pi,\pi]$ and $ \lambda \in [0,1]$. 
For the even function $f_1$, \cite[Lemma 25]{BR18} implies
\[
f_1(\lambda s) \ge f_1(s), \qquad \forall  s \in [-\pi,\pi], \lambda \in [0,1].
\]
Then our proposition follows from the inequalities for $f_1, f_2$ above, Lemma \ref{chaD}, and \eqref{JacCr}.
\end{proof}

\medskip

\subsection*{Acknowledgments}
The study of the sub-Riemannian exponential map in Carnot groups of step 2 as we have developed in Section~\ref{sec6496dd08}, was taught to us by Alessio Martini.
We are very grateful to him.

We want also to thank Luca Rizzi and Kenshiro Tashiro for fruitful discussions. We would also like to thank the anonymous referee for many useful suggestions and valuable remarks which improve the writing of the paper.

SNG would like to thank the Theoretical Sciences Visiting Program (TSVP) at the Okinawa Institute of Science and Technology (OIST) for enabling his visit, and for the generous hospitality and excellent working conditions during his time there, during which part of this work was completed.

\subsection*{Funding information}
SNG has been supported by the Academy of Finland (%
grant 328846, ``Singular integrals, harmonic functions, and boundary regularity in Heisenberg groups'',
grant 322898 ``Sub-Riemannian Geometry via  Metric-geometry and Lie-group Theory'',
grant 314172 ``Quantitative rectifiability in Euclidean and non-Euclidean spaces'').

\subsection*{Conflict of interest}
Authors state no conflict of interest.

\subsection*{Author Contribution}
All authors have accepted responsibility for the entire content of this manuscript and consented to its submission to the journal, reviewed all the results and approved the final version of the manuscript.
SNG and YZ contributed equally to the manuscript. The authors applied the EC norm (the ``equal contribution'' norm) for the sequence of authors.

\subsection*{Data Availability}
Data sharing is not applicable to this article as no datasets were generated or analysed during the current study.



\printbibliography

\end{document}